\documentclass{article}

\usepackage{graphicx} 
\usepackage{amsmath,amssymb,bm,amsthm,mathrsfs,amscd,authblk}
\usepackage{here}

\usepackage{fancyhdr}
\usepackage{tikz}
\usepackage{color}
\theoremstyle{plain}
\newtheorem{Them}{Theorem}[section]
\newtheorem{Lem}[Them]{Lemma}
\newtheorem{Cor}[Them]{Corollary}

\newtheorem{Conj}[Them]{Conjecture}
\newtheorem{Que}[Them]{Question}
\theoremstyle{definition}
\newtheorem{Defi}[Them]{Definition}
\theoremstyle{remark}

\DeclareMathOperator{\lcm}{lcm}
\DeclareMathOperator{\Mat}{Mat}
\DeclareMathOperator{\rank}{rank}

\numberwithin{equation}{section}

\title{Characteristic quasi-polynomials of deletions of Shi arrangements of type B and their period collapse}
\author{Akihiro Higashitani \thanks{Graduate School of Information Science and Techonology, Osaka University, Suita 565-0871, Japan. \\
Email: higashitani@ist.osaka-u.ac.jp} \and Norihiro Nakashima \thanks{Department of Mathematics, Nagoya Institute of Technology, Aichi 466-8555, Japan. \\
Email: nakashima@nitech.ac.jp}}
\date{}

\begin{document}

\maketitle

\begin{abstract}
    Characteristic quasi-polynomials are the enumerative functions counting the number of elements in the complement of hyperplane arrangements modulo positive integers. 
    A notable phenomenon in this context is period collapse, where the quasi-polynomial reduces to a polynomial or has a smaller period than the lcm period.
    In this paper, we compute the characteristic quasi-polynomials of the restriction of the Shi arrangement of type B by one given hyperplane. 
    As a corollary, we completely determine whether period collapse occurs in the characteristic quasi-polynomial of the deletion of the Shi arrangement of type B. 
    This implies the solution for the conjecture posed by Higashitani, Tran and Yoshinaga in this case. 
    

\end{abstract}

\tableofcontents

\section{Introduction}

Quasi-polynomials arise naturally in various areas of mathematics, including Ehrhart theory, representation theory, and the study of arithmetic and combinatorial invariants associated with discrete geometric structures.
One of the most remarkable phenomena in this context is \emph{period collapse}, where a quasi-polynomial exhibits a strictly smaller period than expected, or even reduces to an ordinary polynomial.
Despite its apparent simplicity, period collapse remains poorly understood in general, and only a limited number of explicit and structurally rich examples have been identified so far.
In this paper, we focus on the characteristic quasi-polynomial of hyperplane arrangements and analyze the occurrence of period collapse, offering new contributions to this context of research.

In the theory of hyperplane arrangements, the characteristic polynomial of an arrangement $\mathcal{A}$ is a fundamental invariant encoding deep algebraic and combinatorial information.
In \cite{KTT08,KTT11}, Kamiya, Takemura, and Terao introduced a natural generalization of the characteristic polynomial, called the \emph{characteristic quasi-polynomial}.
More precisely, they showed that for a hyperplane arrangement $\mathcal{A}$ defined over the integers, the number of points in the complement of $\mathcal{A}$ when regarded as a subset of $(\mathbb{Z}/q\mathbb{Z})^m$ is a quasi-polynomial of $q$.
(The precise definition will be recalled in Section~\ref{sec:char}.)
This invariant provides a bridge between hyperplane arrangements and arithmetic aspects of quasi-polynomial behavior.
In fact, recent work by Liu, Tran, and Yoshinaga \cite{Liu-Tran-Yoshinaga21} shows that characteristic quasi-polynomials are closely related to invariants of associated toric arrangements, further clarifying the combinatorial and arithmetic information captured by these invariants.

Since its introduction, the characteristic quasi-polynomial has been computed and studied for several important classes of hyperplane arrangements.
For instance, characteristic quasi-polynomials of extended Linial arrangements play a key role in the proof of a conjecture of Postnikov and Stanley \cite{PostnikovStanley} (see \cite{Yoshinaga2018a} and \cite{Tamura23}).
Characteristic quasi-polynomials of extended Shi arrangements have also been investigated, revealing rich arithmetic and combinatorial structures and providing concrete instances where periodic phenomena can be analyzed explicitly:
\begin{Them}[{Yoshinaga \cite[Theorem 5.1]{Yoshinaga2018}}]\label{thm:Y}
    Let $k \in \mathbb{Z}_{>0}$. The characteristic quasi-polynomial of the extended Shi arrangement $\mathcal{A}_\Phi^{[1-k,k]}$ is equal to $(q-kh)^m$, where $h$ is its Coxeter number. 
\end{Them}

Recall that the Shi arrangement $\mathcal{B}_m$ of type B is defined by
\begin{align*}
\mathcal{B}_m:=\{\{x_i=0\},\{x_i=1\}\mid i\in [m]\}\cup\{\{x_i\pm x_j=0\},\{x_i\pm x_j=1\}\mid 1\leq i<j\leq m\}, 
\end{align*}
where $[m]=\{1,2,\ldots,m\}$. 
In this paper, we study deletions of Shi arrangements associated with the root system of type~$B$.
These arrangements form a natural and tractable family within the broader framework of hyperplane arrangements, where explicit computations of characteristic quasi-polynomials are possible.
Our main results provide closed formulas for these quasi-polynomials and show that, in this setting, period collapse occurs systematically.
Thus, the Shi arrangements of type~$B$ offer a new and computable source of examples illuminating the general phenomenon of period collapse in characteristic quasi-polynomials.

To be precise, we say that a period collapse occurs for a characteristic quasi-polynomial of a hyperplane arrangement $\mathcal{A}$ if its minimum period is strictly smaller than the lcm period. (See Section~\ref{sec:pc} for precise definitions and further details.)
It is shown in \cite{HTY2023} that for any positive integers $m$, $p$ and $s$ such that $s$ divides $p$, there exists a non-central hyperplane $m$-arrangement whose characteristic quasi-polynomial has lcm period $p$ but has minimum period $s$. 
Roughly speaking, this means that ``any possible period collapse'' can occur. 
In addition, Theorem~\ref{thm:Y} shows that the characteristic quasi-polynomial of the extended Shi arrangement is in fact a polynomial, i.e., its minimum period equals $1$.
On the other hand, the lcm period of the extended Shi arrangement of type B is known to be $2$. 
Hence, period collapse occurs in the extended Shi arrangement of type B. 

\medskip

Now, it is natural to think of finding a new example of arrangements in which period collapse occurs. 
The main motivation to organize this paper is to tackle the following naive question: 
\begin{Que}\label{q}
    Let $\mathcal{A}$ be a hyperplane arrangement. Fix $H \in \mathcal{A}$. 
    When does period collapse occur in $\mathcal{A}\setminus \{H\}$? 
\end{Que}

This paper gives a complete solution of this question when $\mathcal{A}=\mathcal{B}_m$, i.e., in the case of the Shi arrangement of type B. 
For this purpose, we compute the characteristic quasi-polynomial of the restrictions $\mathcal{B}_m^H$ of $\mathcal{B}_m$ by a given hyperplane $H \in \mathcal{B}_m$. 
The following theorem is the main result of this paper: 
\begin{Them}[See Theorems~\ref{chara-quasi-del-by-xi=0}, \ref{chara-quasi-del-by-xi=1}, \ref{chara-quasi-del-by-xi=xj}, \ref{chara-quasi-del-by-xi=xj+1}, \ref{chara-quasi-del-by-xi=-xj} and \ref{chara-quasi-del-by-xi=-xj+1}]\label{thm:main}
Fix $i$ and $j$ with $1\leq i<j \leq m$. Then 
\begin{align*}
\left|M((\mathcal{B}_m^{\{x_i=0\}})_q)\right|&= (q-2m+1)^{m-i}(q-2m+2)^{i-1}; \\
\left|M((\mathcal{B}_m^{\{x_i=1\}})_q)\right|&= (q-2m)^{i-1}(q-2m+1)^{m-i}; \\
\left|M((\mathcal{B}_m^{\{x_i-x_j=0\}})_q)\right|&=
\begin{cases}
(q-2m+1)^{j-i-1}(q-2m+2)^{m-j}((q-2m+2)^i-(q-2m+1)^{i-1}) \;\;&\text{if $q$ is odd}, \\
(q-2m+1)^{j-i-1}(q-2m+2)^{i-1}((q-2m+2)^{m-j+1}-(q-2m+1)^{m-j}) &\text{if $q$ is even}; 
\end{cases} \\
\left|M((\mathcal{B}_m^{\{x_i-x_j=1\}})_q)\right|&=
(q-2m)^{m+i-j}(q-2m+1)^{j-i-1}; \\
\left|M((\mathcal{B}_m^{\{x_i+x_j=0\}})_q)\right|&=
\begin{cases}
(q-2m)^{m-j}(q-2m+1)^{j-i}((q-2m+2)^{i-1}-(q-2m+1)^{i-2}) \;\;&\text{if $q$ is odd}, \\
(q-2m)^{m-j+1}(q-2m+1)^{j-i-1}(q-2m+2)^{i-1} &\text{if $q$ is even}; 
\end{cases} \\
\left|M((\mathcal{B}_m^{\{x_i+x_j=1\}})_q)\right|&=
\begin{cases}
(q-2m)^{i-1}(q-2m+1)^{j-i}(q-2m+2)^{m-j} \;\;&\text{if $q$ is odd}, \\
(q-2m)^{i-1}(q-2m+1)^{j-i-1}((q-2m+2)^{m-j+1}-(q-2m+1)^{m-j}) &\text{if $q$ is even}. 
\end{cases}
\end{align*}
\end{Them}

As a corollary of Theorem~\ref{thm:main}, we obtain the following: 
\begin{Cor}\label{period-collapse-1}
Fix $H \in \mathcal{B}_m$. 
Then the characteristic quasi-polynomial of $\mathcal{B}_m \setminus \{H\}$ becomes a polynomial if and only if $H$ is one of the following: 
\begin{itemize}
    \item $H=\{x_i=0\}$ for $1 \leq i \leq m$; 
    \item $H=\{x_i=1\}$ for $1 \leq i \leq m$;
    \item $H=\{x_i-x_{m+1-i}=0\}$ for $1 \leq i \leq m$;
    \item $H=\{x_i-x_j=1\}$ for $1 \leq i<j \leq m$; 
    \item $H=\{x_1+x_j=0\}$ for $2 \leq j \leq m$; 
    \item $H=\{x_i+x_m=1\}$ for $1 \leq i \leq m-1$.  
\end{itemize}
\end{Cor}

As a related question with Question~\ref{q}, the following conjecture was stated: 
\begin{Conj}[{\cite[Conjecture~5.9]{HTY2023}}]
    Let $k \in \mathbb{Z}_{>0}$. If $\delta \in \Phi^+$, then period collapse occurs in $\chi_{\mathrm{Shi}_{\delta^c}^{[1-k,k]}}^\mathrm{quasi}(q)$ in all cases except (i) $\Phi=A_m$, (ii) $\Phi=B_2$ and $\delta$ is a long root. 
\end{Conj}
See \cite[Section 5]{HTY2023} for the notation used in this conjecture. 
In our setting, i.e., in the case of the Shi arrangement of type B, this conjecture predicts that if $m \geq 3$, then 
period collapse always occurs in the characteristic quasi-polynomial of $\mathcal{B}_m \setminus \{H,H^{\prime}\}$, where $H$ and $H^{\prime}$ are hyperplanes in $\mathcal{B}_m$ which are parallel each other. 
We see that this scarcely holds in general as the following corollary shows: 
\begin{Cor}\label{period-collapse-2}
Fix $H,H^{\prime} \in \mathcal{B}_m$ which are parallel each other. 
Then the characteristic quasi-polynomial of $\mathcal{B}_m \setminus \{H,H^{\prime}\}$ becomes a polynomial if and only if one of the following is satisfied (by replacing $H$ and $H^{\prime}$ if necessary): 
\begin{itemize}
    \item $H=\{x_i=0\}$ and $H^{\prime}=\{x_i=1\}$ for $1 \leq i \leq m$; 
    \item $H=\{x_i-x_{m+1-i}=0\}$ and $H^{\prime}=\{x_i-x_{m+1-i}=1\}$ for $1 \leq i \leq m$; 
    \item $H=\{x_i+x_{m+1-i}=0\}$ and $H^{\prime}=\{x_i+x_{m+1-i}=1\}$ for $1 \leq i \leq m$. 
\end{itemize}
\end{Cor}

\bigskip

A brief organization of this paper is as follows. 
In Section~\ref{sec:pre}, we recall the details of some notions, e.g., characteristic quasi-polynomials and period collapse. Moreover, we give another proof of Theorem~\ref{thm:Y} in the case of the Shi arrangement of type B in Section~\ref{sec-count-method} for the use of our proofs of the main theorem by modifying it. 
In Section~\ref{sec-quasi-poly-restxi=1}, the characteristic quasi-polynomial of the restriction of $\mathcal{B}_m$ by the hyperplane $\{x_i=1\}$ is determined (as well as the case of the hyperplane $\{x_i=0\}$). Namely, Section~\ref{sec-quasi-poly-restxi=1} gives a proof of the first two equalities of Theorem~\ref{thm:main}. 
Similarly, in Sections~\ref{sec-quasi-poly-restxi=xj}, \ref{sec-quasi-poly-restxi=xj+1}, \ref{sec-quasi-poly-restxi=-xj} and \ref{sec-quasi-poly-restxi=-xj+1}, the characteristic quasi-polynomial of the restriction of $\mathcal{B}_m$ by the hyperplane $\{x_i-x_j=0\}$, $\{x_i-x_j=1\}$, $\{x_i+x_j=0\}$ and $\{x_i+x_j=1\}$ is determined, respectively. 
Finally, in Section 8, we give proofs of Corollaries~\ref{period-collapse-1} and \ref{period-collapse-2}. 

\section*{Acknowledgements}
A.H. is partially supported by KAKENHI JP21KK0043 and JP23H00081.
N.N. is partially supported by KAKENHI JP20K14282.
\section{Preliminaries}\label{sec:pre}

\subsection{Characteristic quasi-polynomial}\label{sec:char}
A function $\chi:\mathbb{Z}_{>0}\rightarrow\mathbb{Z}$ is called a \textit{quasi-polynomial} if there exist a positive integer $\rho\in\mathbb{Z}_{>0}$ and polynomials $\chi^1(t),\chi^2(t),\dots,\chi^{\rho}(t)\in\mathbb{Z}[t]$ such that for $q\in\mathbb{Z}_{>0}$,
\begin{align*}
\chi(q)=
\begin{cases}
\chi^1(q)\quad&\text{if }q\equiv 1\  \mod{\rho},\\
\chi^2(q)\quad&\text{if }q\equiv 2\  \mod{\rho},\\
\quad\vdots&\quad\vdots\\
\chi^{\rho}(q)\quad&\text{if }q\equiv \rho\ \mod{\rho}.\\
\end{cases}
\end{align*}
The number $\rho$ is called a \textit{period} and the polynomials $\chi^1(t),\chi^2(t),\dots,\chi^{\rho}(t)$ are called the \textit{$k$-constituents}.
A quasi-polynomial $\chi$ with a period $\rho$ is said to have the \textit{gcd property} with respect to $\rho$ if the $k$-constituents $\chi^1(t),\chi^2(t),\dots,\chi^{\rho}(t)$ depend on $k$ only through $\gcd(\rho,k)$, i.e., $\chi^a(t)=\chi^b(t)$ if $\gcd(\rho,a)=\gcd(\rho,b)$.
We call $\chi$ \textit{monic} if all of $\chi^1(t),\chi^2(t),\dots,\chi^{\rho}(t)$ are monic polynomials. 

Let $q\in\mathbb{Z}_{>0}$ and define $\mathbb{Z}_q:=\mathbb{Z}/q\mathbb{Z}$.
For $a\in\mathbb{Z}$, let $[a]_q:=a+q\mathbb{Z}\in\mathbb{Z}_q$ be the $q$ reduction of $a$.
An $m\times n$ matrix $A=(a_{ij})\in\Mat_{m\times n}(\mathbb{Z})$ with integral entries and an integral vector $\bm{b}=(b_1,\dots,b_n)\in\mathbb{Z}^n$ define a hyperplane arrangement $\mathcal{A}_q=\{H_{1,q},\dots,H_{n,q}\}$ over $\mathbb{Z}_q$, where $H_{i,q}=\{(x_1,\dots,x_m)\in\mathbb{Z}_q\mid [a_{1,i}]_q x_1+\cdots+[a_{m,i}]_q x_m=[b_i]_q\}$ for any $i\in [n]$.
In addition, let $M(\mathcal{A}_q)$ be the complement of $\mathcal{A}_q$, i.e., 
\begin{align*}
M(\mathcal{A}_q):=\mathbb{Z}_q^m\setminus\bigcup_{i=1}^n H_{i,q}
=\{(x_1,\dots,x_m)\in\mathbb{Z}_q^m\mid [a_{1,i}]_q x_1+\cdots+[a_{m,i}]_q x_m\neq [b_i]_q\ \text{for any}\ i\in [n]\}.
\end{align*}
For $\emptyset\neq J\subseteq [n]$, the matrix $A_J\in\Mat_{m\times|J|}(\mathbb{Z})$ is defined by the submatrix of $A$ consisting of the columns indexed by $J$.
Set $\ell(J):=\rank(A_J)$.
Let $e_{J,1}|e_{J,2}|\cdots|e_{J,\ell(J)}$ be the elementary divisors of $A_J$.
Then we define $\rho_A:=\lcm\left(e_{J,\ell(J)}\,\middle|\,\emptyset\neq J\subseteq [n]\right)$.
\begin{Them}[Kamiya, Takemura, and Terao \cite{KTT08, KTT11}]\label{thm-charquasi-KTT}
The function $|M(\mathcal{A}_q)|$ is a monic quasi-polynomial with a period $\rho_{A}$ having the gcd property with respect to $\rho_{A}$.
\end{Them}
The quasi-polynomial in Theorem \ref{thm-charquasi-KTT} is called the \textit{characteristic quasi-polynomial} and $\rho_A$ is called the \textit{lcm period} of $|M(\mathcal{A}_q)|$.
In addition, the minimum value of the periods of $|M(\mathcal{A}_q)|$ is called the \textit{minimum period}, denoted by $\rho_\mathrm{min}$.

\subsection{Period collapse}\label{sec:pc}

The following is an interesting result for the minimum period and the lcm period.
\begin{Them}[Higashitani, Tran, and Yoshinaga \cite{HTY2023}]\label{thm-lcm=minimum}
We have the following.
\begin{itemize}
\item[$(1)$]\ For any $s,p\in\mathbb{Z}_{>0}$ with $s|p$, there exist a matrix $A$ and a vector $\bm{b}\neq\bm{0}$ such that 
$\rho_\mathrm{min}=s$ and $\rho_A=p$.
Especially, $\rho_\mathrm{min}<\rho_A$ holds.
\item[$(2)$]\ If $\bm{b}=\bm{0}$ (this case is called the central case), then the lcm period coincides with the minimum period, that is, $\rho_\mathrm{min}=\rho_A$.
\end{itemize}
\end{Them}
From Theorem \ref{thm-lcm=minimum}, the question arises as to when the minimum period and the lcm period coincide.
To consider this question, we define period collapse as follows.
\begin{Defi}
\textit{Period collapse} occurs in $|M(\mathcal{A}_q)|$ if $\rho_\mathrm{min}<\rho_A$.
\end{Defi}
Note that period collapse does not occur in the central case as Theorem~\ref{thm-lcm=minimum} (2) claims.
We give two remarks about the characteristic quasi-polynomial and the lcm period for the Shi arrangements.
\begin{itemize}
\item An $m\times m$ matrix $P$ with integral entries is said to be unimodular if $\det(P)\in\{1,-1\}$.
The characteristic quasi-polynomial and the lcm period are invariant under the action of multiplying a unimodular matrix from the left.
In \cite{HTY2023}, the coefficients of the matrix defining the Shi arrangement are written as the linear combinations of the simple root basis. 
On the other hand, the coefficients of the matrix defining the Shi arrangement in this paper are the linear combination of the standard basis.
Since these definitions are transformed by the unimodular matrix, we are considering the same object.
\item It is easy to see that the lcm period of the matrix $A$ defining $\mathcal{B}_m$ and matrices obtained by removing columns from $A$ is always less than or equal to $2$.
In other words, when computing the characteristic quasi-polynomial of the Shi arrangement of type B or its deletions, we only need to consider the cases that $q$ is even or odd.
In addition, when the lcm period is $2$, period collapse occurs in $|M(\mathcal{A}_q)|$ if and only if $|M(\mathcal{A}_q)|$ is a polynomial.
\end{itemize}

\subsection{Equivalence of deletions and restrictions}\label{subsec:equiv}
Let $A=(a_{ij})$ be an $m\times n$ matrix with integral entries and $\bm{b}=(b_1,\dots,b_n)$ an integral vector. 
Let $\mathcal{A}$ be the hyperplane arrangement defined by $(A,\bm{b})$, and let $H_q=\{(x_1,\dots,x_m)\in\mathbb{Z}_q\mid [a_{1}]_q x_1+\cdots+[a_{m}]_q x_m=[b]_q\}\in\mathcal{A}_q$.
For $H_q\in\mathscr{A}_q$, we call $\mathscr{A}_q\setminus\{H_q\}$ the \textit{deletion} and
$\mathscr{A}_q^{H_q}:=\left\{H_q\cap H_{q}^{\prime}\,\middle|\,H_{q}^{\prime}\in\mathscr{A}_q\setminus\{H_q\},\ H_q\cap H_{q}^{\prime}\neq\emptyset\right\}$ the \textit{restriction}.
If there exists a unimodular matrix $P\in\Mat_{m\times m}(\mathbb{Z})$ such that $P\,^t([a_1]_q,\dots,[a_m]_q)=\,^t([1]_q,[0]_q,\dots,[0]_q)$, then the following formula holds: 
\begin{align*}
    |M(\mathcal{A}_q)|=\left|M\left(\mathcal{A}_q\setminus \{H_{q}\}\right)\right| - \left|M\left(\mathcal{A}_q^{H_{q}}\right)\right|. 
\end{align*}
See, e.g., \cite[Corollary 4.11]{Liu-Tran-Yoshinaga21} or \cite[Corollary 4.2]{MN2024}. 
This equality says that if period collapse occurs in two of $|M(\mathcal{A}_q)|$, $\left|M\left(\mathcal{A}_q\setminus \{H_{q}\}\right)\right|$ and $\left|M\left(\mathcal{A}_q^{H_{q}}\right)\right|$, then period collapse occurs also in the remaining one. 
In the case of $\mathcal{A}=\mathcal{B}_m$, the lcm periods is $2$, and any hyperplane $H\in\mathcal{A}$ is transferred to $\{x_1=0\}$ by unimodular transformation.
Since we know that $\left|M\left((\mathcal{B}_m)_q\right)\right|$ is a polynomial by Theorem~\ref{thm:Y}, the fact that $\left|M\left((\mathcal{B}_m)_q\setminus \{H_{q}\}\right)\right|$ is a polynomial and the fact that $\left|M\left((\mathcal{B}_m^H)_q\right)\right|$ is a polynomial are equivalent. 
Note that since there is an index $i$ such that the coefficient of $x_i$ is $1$ in the defining equation of $H\in\mathcal{B}_m$, we have $\left|M\left(\left(\mathcal{B}_m^H\right)_q\right)\right|=\left|M\left((\mathcal{B}_m)_q^{H_q}\right)\right|$, where $H_q$ is the $q$-reduction of $H$ and $\mathcal{B}_m^H$ is the restriction over $\mathbb{R}^m$. 
Therefore, by checking whether the characteristic quasi-polynomial of a restriction is a polynomial, we can check whether the characteristic quasi-polynomial of a deletion is a polynomial. 

\subsection{Counting method of the elements in the complement of the Shi arrangement}\label{sec-count-method}
In the rest of this section, let $\mathcal{A}=\mathcal{B}_m$.
We recall that the characteristic quasi-polynomial of the Shi arrangement of type B is computed as follows:
\begin{align*}
|M(\mathcal{A}_q)|=(q-2m)^m\ \text{for any}\ q\gg 0.
\end{align*}
Athanasiadis \cite[Theorem 3.10 and Theorem 3.14]{Athanasiadis1996} gave a combinatorial method to count the the elements in $M(\mathcal{A}_q)$ and prove the equality when $q$ is a prime.
In this paper, we modify this method in several ways to compute the characteristic quasi-polynomial of the arrangement formed by deleting a single hyperplane from the Shi arrangement.
In particular, we extend this method to the case when $q$ is even as well as the case when $q$ is odd.
For this reason, we first describe the details of this method.

Recall that the complement $M(\mathcal{A}_q)$ is the set of $(x_1,\dots,x_m)\in\mathbb{Z}_q^m$ satisfying the following conditions:
\begin{align*}
&x_s\neq 0,\ x_s\neq 1\ (s\in [m]),\\
&x_s\neq x_t\ (s,t\in [m],\ s\neq t),\\
&x_s\neq x_t+1\ (s,t\in [m],\ s<t),\\
&x_s\neq -x_t,\ x_s\neq -x_t+1\ (s,t\in [m],\ s\neq t).
\end{align*}

\subsubsection{The case when $q$ is odd}\label{sec-count-method-odd}
In this subsection, we assume that $q$ is odd while we describe the counting method by Athanasiadis \cite{Athanasiadis1996}.
To aid understanding, at each step we add an example for $m=5$ and $q=15$.
\begin{itemize}
\item[Step 1.] Prepare $q-2m+1$ boxes side by side, $\displaystyle \frac{q+1}{2}-m$ on the upper side and $\displaystyle \frac{q+1}{2}-m$ on the lower side.
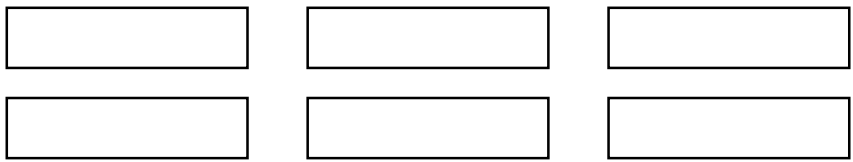
\begin{center}
\scalebox{0.8}{
\begin{tikzpicture}
\draw [very thick] (0,0) rectangle (4,1); \draw [very thick] (5,0) rectangle (9,1); \draw [very thick] (10,0) rectangle (14,1);
\draw [very thick] (0,1.5) rectangle (4,2.5); \draw [very thick] (5,1.5) rectangle (9,2.5); \draw [very thick] (10,1.5) rectangle (14,2.5);
\end{tikzpicture}
}
\end{center}
\item[Step 2.] Place each of the numbers $1,\dots,m$ corresponding to the indices of $x_1,\dots,x_m$ in one of $q-2m$ boxes, avoiding the upper left box.
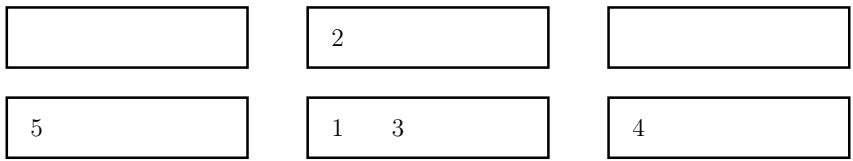
\begin{center}
\scalebox{0.8}{
\begin{tikzpicture}[main/.style = {draw, circle, very thick}] 
\draw [very thick] (0,0) rectangle (4,1); \draw [very thick] (5,0) rectangle (9,1); \draw [very thick] (10,0) rectangle (14,1);
\draw [very thick] (0,1.5) rectangle (4,2.5); \draw [very thick] (5,1.5) rectangle (9,2.5); \draw [very thick] (10,1.5) rectangle (14,2.5);
\node(1) at (5.5,2) {{\large $2$}}; \node(2) at (10.5,0.5) {{\large $4$}}; \node(3) at (0.5,0.5) {{\large $5$}}; \node(4) at (5.5,0.5) {{\large $1$}}; \node(5) at (6.5,0.5) {{\large $3$}};
\end{tikzpicture}
}
\end{center}

\item[Step 3.] Place unlabeled circles at the left edges of all boxes except the lower left box.
The number in each of the upper boxes is placed in ascending order from left to right next to the unlabeled circle.
Rewrite each number as a circle labeled with the same number.
Place the unlabeled circles on the opposite side of the labeled circles arranged in this manner.
Then the number in each of the lower boxes is placed in descending order, starting from next of unlabeled circles placed in this way.
Also rewrite each number as a circle labeled with the same number and place the unlabeled circles on the opposite side.
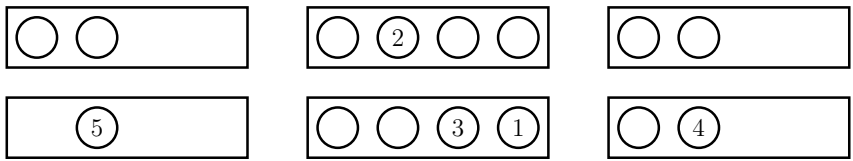
\begin{center}
\scalebox{0.8}{
\begin{tikzpicture}[main/.style = {draw, circle, very thick}] 
\draw [very thick] (0,0) rectangle (4,1); \draw [very thick] (5,0) rectangle (9,1); \draw [very thick] (10,0) rectangle (14,1);
\draw [very thick] (0,1.5) rectangle (4,2.5); \draw [very thick] (5,1.5) rectangle (9,2.5); \draw [very thick] (10,1.5) rectangle (14,2.5);
\node[main](1) at (0.5,2) {{\color{white}\large $0$}}; \node[main](2) at (1.5,2) {{\color{white}\large $0$}};
\node[main](3) at (5.5,2) {{\color{white}\large $0$}}; \node[main](4) at (6.5,2) {{\large $2$}};\node[main](5) at (7.5,2) {{\color{white}\large $0$}}; \node[main](6) at (8.5,2) {{\color{white}\large $0$}};
\node[main](7) at (10.5,2) {{\color{white}\large $0$}}; \node[main](8) at (11.5,2) {{\color{white}\large $0$}};
\node[main](9) at (1.5,0.5) {{\large $5$}};
\node[main](10) at (5.5,0.5) {{\color{white}\large $0$}}; \node[main](11) at (6.5,0.5) {{\color{white}\large $0$}};\node[main](12) at (7.5,0.5) {{\large $3$}}; \node[main](13) at (8.5,0.5) {{\large $1$}};
\node[main](14) at (10.5,0.5) {{\color{white}\large $0$}}; \node[main](15) at (11.5,0.5) {{\large $4$}};
\end{tikzpicture}
}
\end{center}

\item[Step 4.] Arrange the circles clockwisely starting from the circle at the left end of the upper left box.
In the following, the black filled circle corresponds to the circle at the left end of the upper left box.

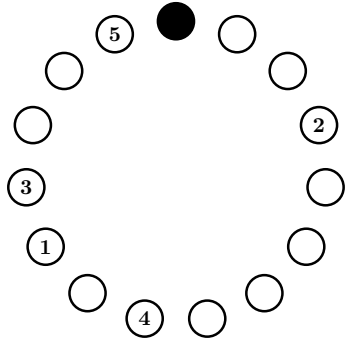
\begin{center}
\scalebox{0.8}{
\begin{tikzpicture}[every edge quotes/.style = {auto, font=\footnotesize, sloped}]
\begin{scope}[rotate=90, transform shape]
\foreach \i in {1,...,15} {
    \path[draw=black, very thick] (0,3) ++({360/15 * (\i - 1)}:2.5) coordinate (n\i)
        circle[radius=3mm];
    \node (n\i) []  at (n\i) {};
}

\node[circle, radius=3mm, fill=black] at (n1) {0};
\node[font=\bfseries, rotate=-90] at (n6) {1};
\node[font=\bfseries, rotate=-90] at (n13) {2};
\node[font=\bfseries, rotate=-90] at (n5) {3};
\node[font=\bfseries, rotate=-90] at (n8) {4};
\node[font=\bfseries, rotate=-90] at (n2) {5};

\end{scope}
\end{tikzpicture}
}
\end{center}

\item[Step 5.] The first circle placed in Step 4 (the black circle above) corresponds to $0$, and other circles clockwisely correspond to the elements in $\{1,\dots,q-1\}$ in ascending order.
Create a tuple whose $i$-th entry is the element in $\{1,\dots,q-1\}$ corresponding to the circle labeled with $i$.
Then we obtain an element $(x_1,\dots,x_m)\in\mathbb{Z}_q^m$.
In the case of the example above, $(x_1,x_2,x_3,x_4,x_5)=(10,3,11,8,14)$.
We can also consider the assigninment of the elements in $\mathbb{Z}_q$ to the boxes and circles as follows.
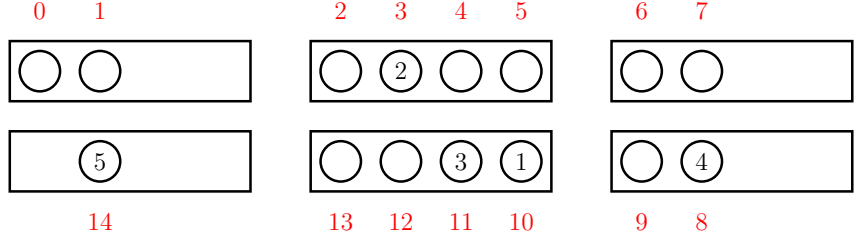
\begin{center}
\scalebox{0.8}{
\begin{tikzpicture}[main/.style = {draw, circle, very thick}] 
\draw [very thick] (0,0) rectangle (4,1); \draw [very thick] (5,0) rectangle (9,1); \draw [very thick] (10,0) rectangle (14,1);
\draw [very thick] (0,1.5) rectangle (4,2.5); \draw [very thick] (5,1.5) rectangle (9,2.5); \draw [very thick] (10,1.5) rectangle (14,2.5);
\node[main](1) at (0.5,2) {{\color{white}\large $0$}}; \node[main](2) at (1.5,2) {{\color{white}\large $0$}};
\node[main](3) at (5.5,2) {{\color{white}\large $0$}}; \node[main](4) at (6.5,2) {{\large $2$}};\node[main](5) at (7.5,2) {{\color{white}\large $0$}}; \node[main](6) at (8.5,2) {{\color{white}\large $0$}};
\node[main](7) at (10.5,2) {{\color{white}\large $0$}}; \node[main](8) at (11.5,2) {{\color{white}\large $0$}};
\node[main](9) at (1.5,0.5) {{\large $5$}};
\node[main](10) at (5.5,0.5) {{\color{white}\large $0$}}; \node[main](11) at (6.5,0.5) {{\color{white}\large $0$}};\node[main](12) at (7.5,0.5) {{\large $3$}}; \node[main](13) at (8.5,0.5) {{\large $1$}};
\node[main](14) at (10.5,0.5) {{\color{white}\large $0$}}; \node[main](15) at (11.5,0.5) {{\large $4$}};
\node(0) at (0.5,3) {{\color{red}\large $0$}}; \node(1) at (1.5,3) {{\color{red}\large $1$}}; \node(2) at (5.5,3) {{\color{red}\large $2$}}; \node(3) at (6.5,3) {{\color{red}\large $3$}}; \node(4) at (7.5,3) {{\color{red}\large $4$}}; \node(5) at (8.5,3) {{\color{red}\large $5$}}; \node(6) at (10.5,3) {{\color{red}\large $6$}}; \node(7) at (11.5,3) {{\color{red}\large $7$}};
\node(8) at (11.5,-0.5) {{\color{red}\large $8$}}; \node(9) at (10.5,-0.5) {{\color{red}\large $9$}}; \node(10) at (8.5,-0.5) {{\color{red}\large $10$}}; \node(11) at (7.5,-0.5) {{\color{red}\large $11$}}; \node(12) at (6.5,-0.5) {{\color{red}\large $12$}}; \node(13) at (5.5,-0.5) {{\color{red}\large $13$}}; \node(14) at (1.5,-0.5) {{\color{red}\large $14$}};
\end{tikzpicture}
}
\end{center}

\end{itemize}

The boxes and circles in Step 3 can be identified with the circles in Step 4, and they have the following properties:

\begin{itemize}
\item The circles corresponding to the elements $0,1\in\mathbb{Z}_q$ have no labels since we avoid putting the numbers $1,\dots,m$ in the upper left box.
This corresponds to the condition $x_i\neq 0$ and $x_i\neq 1$ for any $i\in [m]$.
\item Each circle is labeled with at most one number.
This corresponds to the condition that $x_i\neq x_j$ for any $i,j\in [m]$ with $i\neq j$.
\item There is always an unlabeled circle on the opposite side of the labeled circle.
This corresponds to the condition that $x_i\neq -x_j$ for any $i,j\in [m]$ with $i\neq j$.
\item The circle preceding the circle labeled with $i$ in a clockwise direction does not have a label greater than $i$.
This corresponds to the condition that $x_i\neq x_j+1$ for any $i,j\in [m]$ with $i<j$.
\item The clockwise next circle from the opposite circle of the labeled circle with $j$ is either an unlabeled circle or the circle labeled with $j$ (i.e., the same circle as the original).
This corresponds to the condition that $x_j\neq -x_i+1$ for any $i,j\in [m]$ with $i\neq j$.
\end{itemize}
From the conditions above, the tuple $(x_1,\dots,x_m)\in\mathbb{Z}_q^m$ obtained in Step 5 is contained in $M(\mathcal{A}_q)$.
Also, following this procedure in reverse, taking an element in $M(\mathcal{A}_q)$ corresponds to placing the numbers $1,\dots,m$ in the $q-2m$ boxes in Step 2.
Therefore we have that $|M(\mathcal{A}_q)|=(q-2m)^m$.

\subsubsection{The case when $q$ is even}\label{sec-count-method-even}
In this subsection, we extend the counting method to the case when $q$ is even, so we assume that $q$ is even.
We count the number of elements $(x_1,\dots,x_m)\in M(\mathcal{A}_q)$ by dividing the cases by whether there exists an index $k$ such that $\displaystyle x_k=\frac{q}{2}$ or not.

(i)\ Suppose that there is no index $k\in[m]$ such that $\displaystyle x_k=\frac{q}{2}$.
Prepare $q-2m$ boxes, $\displaystyle \frac{q}{2}-m$ on the upper side and $\displaystyle \frac{q}{2}-m$ on the lower side, and one circle corresponding to the element $\displaystyle \frac{q}{2}\in\mathbb{Z}_q$ on the right side of the boxes.
Put each of the numbers $1,\dots,m$ into one of the $q-2m-1$ boxes, except the upper left box.
We create $(x_1,\dots,x_m)$ by the same procedure described in Section \ref{sec-count-method-odd}.
Then the element $(x_1,\dots,x_m)$ created in this maner one-to-one corresponds to the element in $M(\mathcal{A}_q)$ that there is no index $k\in[m]$ such that $\displaystyle x_k=\frac{q}{2}$.
Therefore, in this case, the desired number is $(q-2m-1)^m$.
For example, if $m=5$ and $q=16$, then $(x_1,x_2,x_3,x_4,x_5)=(2,9,10,13,11)$ corresponds to the following boxes and circles:
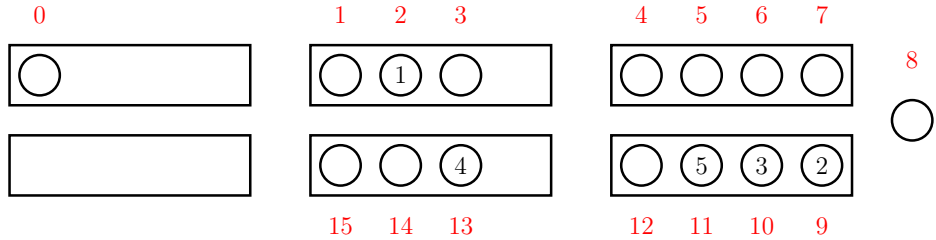
\begin{center}
\scalebox{0.8}{
\begin{tikzpicture}[main/.style = {draw, circle, very thick}] 
\draw [very thick] (0,0) rectangle (4,1); \draw [very thick] (5,0) rectangle (9,1); \draw [very thick] (10,0) rectangle (14,1);
\draw [very thick] (0,1.5) rectangle (4,2.5); \draw [very thick] (5,1.5) rectangle (9,2.5); \draw [very thick] (10,1.5) rectangle (14,2.5);
\node[main](0) at (15,1.25) {{\color{white}\large $0$}};
\node[main](1) at (0.5,2) {{\color{white}\large $0$}};
\node[main](2) at (5.5,2) {{\color{white}\large $0$}}; \node[main](3) at (6.5,2) {{\large $1$}};\node[main](4) at (7.5,2) {{\color{white}\large $0$}};
\node[main](5) at (10.5,2) {{\color{white}\large $0$}}; \node[main](6) at (11.5,2) {{\color{white}\large $0$}}; \node[main](7) at (12.5,2) {{\color{white}\large $0$}}; \node[main](8) at (13.5,2) {{\color{white}\large $0$}};
\node[main](9) at (5.5,0.5) {{\color{white}\large $0$}}; \node[main](10) at (6.5,0.5) {{\color{white}\large $0$}}; \node[main](11) at (7.5,0.5) {{\large $4$}};
\node[main](12) at (10.5,0.5) {{\color{white}\large $0$}}; \node[main](13) at (11.5,0.5) {{\large $5$}}; \node[main](14) at (12.5,0.5) {{\large $3$}}; \node[main](12) at (13.5,0.5) {{\large $2$}};
\node(0) at (0.5,3) {{\color{red}\large $0$}}; \node(1) at (5.5,3) {{\color{red}\large $1$}}; \node(2) at (6.5,3) {{\color{red}\large $2$}}; \node(3) at (7.5,3) {{\color{red}\large $3$}}; \node(4) at (10.5,3) {{\color{red}\large $4$}}; \node(5) at (11.5,3) {{\color{red}\large $5$}}; \node(6) at (12.5,3) {{\color{red}\large $6$}}; \node(7) at (13.5,3) {{\color{red}\large $7$}};
\node(8) at (15,2.25) {{\color{red}\large $8$}}; 
\node(9) at (13.5,-0.5) {{\color{red}\large $9$}}; \node(10) at (12.5,-0.5) {{\color{red}\large $10$}}; \node(11) at (11.5,-0.5) {{\color{red}\large $11$}}; \node(12) at (10.5,-0.5) {{\color{red}\large $12$}}; \node(13) at (7.5,-0.5) {{\color{red}\large $13$}}; \node(14) at (6.5,-0.5) {{\color{red}\large $14$}}; \node(15) at (5.5,-0.5) {{\color{red}\large $15$}};
\end{tikzpicture}
}
\end{center}

(i\hspace{-0.5mm}i)\ Suppose that there exists an index $k\in[m]$ such that $\displaystyle x_k=\frac{q}{2}$.
Create boxes and circles according to the following rules and take the corresponding element $(x_1,\dots,x_m)\in M(\mathcal{A}_q)$.
\begin{itemize}
\item Prepare $q-2m+2$ boxes, $\displaystyle \frac{q}{2}-(m-1)$ on the upper side and $\displaystyle \frac{q}{2}-(m-1)$ on the lower side since $m-1$ numbers make $2(m-1)$ circles in the boxes, and one circle labeled with $k$ on the right side of the boxes.
\item Place each of the numbers $1,\dots,k-1,k+1,\dots,m$ in one of $q-2m+1$ boxes, except the upper left box.
For example, if $m=5$ and $q=16$, then $(x_1,x_2,x_3,x_4,x_5)=(14,15,8,5,10)$ corresponds to the following boxes and circles:
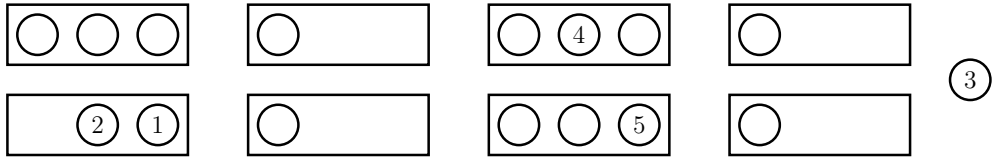
\begin{center}
\scalebox{0.8}{
\begin{tikzpicture}[main/.style = {draw, circle, very thick}] 
\draw [very thick] (0,0) rectangle (3,1); \draw [very thick] (4,0) rectangle (7,1); \draw [very thick] (8,0) rectangle (11,1); \draw [very thick] (12,0) rectangle (15,1);
\draw [very thick] (0,1.5) rectangle (3,2.5); \draw [very thick] (4,1.5) rectangle (7,2.5); \draw [very thick] (8,1.5) rectangle (11,2.5); \draw [very thick] (12,1.5) rectangle (15,2.5);
\node[main](0) at (16,1.25) {{\large $3$}};
\node[main](1) at (0.5,2) {{\color{white}\large $0$}}; \node[main](2) at (1.5,2) {{\color{white}\large $0$}}; \node[main](3) at (2.5,2) {{\color{white}\large $0$}};
\node[main](4) at (4.5,2) {{\color{white}\large $0$}};
\node[main](5) at (8.5,2) {{\color{white}\large $0$}}; \node[main](6) at (9.5,2) {{\large $4$}}; \node[main](7) at (10.5,2) {{\color{white}\large $0$}};
\node[main](8) at (12.5,2) {{\color{white}\large $0$}};
\node[main](9) at (1.5,0.5) {{\large $2$}}; \node[main](10) at (2.5,0.5) {{\large $1$}};
\node[main](11) at (4.5,0.5) {{\color{white}\large $0$}};
\node[main](12) at (8.5,0.5) {{\color{white}\large $0$}}; \node[main](13) at (9.5,0.5) {{\color{white}\large $0$}}; \node[main](14) at (10.5,0.5) {{\large $5$}};
\node[main](15) at (12.5,0.5) {{\color{white}\large $0$}};
\end{tikzpicture}
}
\end{center}
\item The lower right box does not contain any number.
Indeed, if there is a labeled circle in the lower right box, then the clockwise next circle from the opposite circle of the rightmost circle labeled with $i$ is the circle labeled with $k$.
In other words, $x_k=-x_i+1$ holds, which is contradiction.
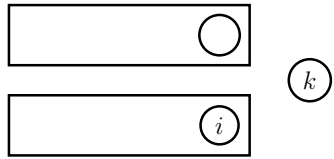
\begin{center}
\scalebox{0.8}{
\begin{tikzpicture}[main/.style = {draw, circle, very thick}] 
\draw [very thick] (0,0) rectangle (4,1); \draw [very thick] (0,1.5) rectangle (4,2.5); \node[main](0) at (5,1.25) {{\large $k$}};
\node[main](2) at (3.5,2) {{\color{white}\large $0$}};
\node[main](4) at (3.5,0.5) {{\large $i$}};
\end{tikzpicture}
}
\end{center}
\item The upper right box does not contain any number larger than $k$.
Indeed, if the rightmost label on the upper right box is $j$ with $k<j$, then the circle clockwise preceding the circle labeled with $k$ has the label greater than $k$.
In other words, $x_k=x_j+1$ holds, which is contradiction.
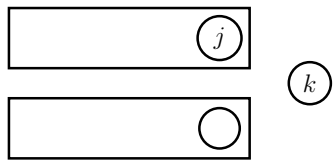
\begin{center}
\scalebox{0.8}{
\begin{tikzpicture}[main/.style = {draw, circle, very thick}] 
\draw [very thick] (0,0) rectangle (4,1); \draw [very thick] (0,1.5) rectangle (4,2.5); \node[main](0) at (5,1.25) {{\large $k$}};
\node[main](2) at (3.5,2) {{\large $j$}};
\node[main](4) at (3.5,0.5) {{\color{white}\large $0$}};
\end{tikzpicture}
}
\end{center}
\end{itemize}
There are $q-2m$ boxes that can contain the numbers $1.\dots,k-1$ and $q-2m-1$ boxes that can contain the numbers $k+1,\dots,m$.
Therefore the desired number is $(q-2m-1)^{m-k}(q-2m)^{k-1}$ for each fixed $k\in [m]$.

Here, we claim the following lemma, which will be often used in the calculation of characteristic quasi-polynomials. 
\begin{Lem}\label{hodai}
Given non-negative integers $a,b$ with $a\leq b$, we have 
\[
\sum_{k=a}^b X^{b-k}(X+1)^{k-a}=\sum_{k=a}^b (X+1)^{b-k}X^{k-a}=(X+1)^{b-a+1}-X^{b-a+1}. 
\]
\end{Lem}
\begin{proof}
We obtain this as follows: 
\begin{align*}
    \sum_{k=a}^b (X+1)^{b-k}X^{k-a}=(X+1-X)\sum_{k=0}^{b-a} (X+1)^{b-a-k}X^{k} 
    =(X+1)^{b-a+1}-X^{b-a+1}.
\end{align*}
\end{proof}
From the discussion above, by applying Lemma~\ref{hodai}, we have that
\begin{align*}
\left|M(\mathcal{A}_q)\right|&=(q-2m-1)^m+\sum_{k=1}^m(q-2m-1)^{m-k}(q-2m)^{k-1}\\
&=(q-2m-1)^m+(q-2m)^{m}-(q-2m-1)^{m}=(q-2m)^{m}.
\end{align*}

\section{Characteristic quasi-polynomial of restriction on $\{x_i=0\}$}\label{sec-quasi-poly-restxi=0}

Let $\mathcal{A}^{(1)}=\mathcal{B}_m^{\{x_i=0\}}$ for $i\in [m]$, and we prove the first equality of Theorem \ref{thm:main} as follows.
\begin{Them}\label{chara-quasi-del-by-xi=0}
We have
\begin{align*}
\left|M(\mathcal{A}^{(1)}_q)\right|=(q-2m+1)^{m-i}(q-2m+2)^{i-1}
\end{align*}
for any $q\in\mathbb{Z}$ with $q\gg 0$.
\end{Them}

The complement $M(\mathcal{A}^{(1)}_q)$ is the set of $(x_1,\dots,x_m)\in\mathbb{Z}_q^m$ satisfying the following conditions:
\begin{align*}
&x_i=0,\\
&x_s\neq 0,\ x_s\neq 1\ (s\in [m],\ s\neq i),\\
&x_s\neq x_t\ (s,t\in [m],\ s\neq t),\\
&x_s\neq x_t+1\ (s,t\in [m],\ s<t),\\
&x_s\neq -x_t,\ x_s\neq -x_t+1\ (s,t\in [m],\ s\neq t).
\end{align*}
We fix the condition $x_i=0$ and count the number of elements in $(x_1,\dots,x_m)\in M(\mathcal{A}^{(1)}_q)$ using a modified version of the counting method described in Section \ref{sec-count-method}.
The condition $x_i=0$ means that the circle at the left edge of the upper left box is the circle labeled with $i$.
The remaining parts of this section are devoted to giving the proof of Theorem~\ref{chara-quasi-del-by-xi=0}. 

\subsection{The proof of Theorem \ref{chara-quasi-del-by-xi=0} when $q$ is odd}\label{sec-quasi-poly-rest-xi=0-odd}
We assume that $q$ is odd.
Making a modification to the procedure in Section \ref{sec-count-method-odd}, we create boxes and circles and take the corresponding element $(x_1,\dots,x_m)\in M(\mathcal{A}^{(1)}_q)$.
\begin{itemize}
\item Prepare $q-2m+3$ boxes side by side, $\displaystyle \frac{q+1}{2}-(m-1)$ on the upper side and $\displaystyle \frac{q+1}{2}-(m-1)$ on the lower side since $m-1$ numbers make $2(m-1)$ circles in the boxes.
Place each of the numbers $1,\dots,i-1,i+1,\dots,m$ into one of $q-2m+3$ boxes.
Create circles by the same procedure described in Section \ref{sec-count-method-odd}, where the circle at the left edge of the upper left box is labeled with $i$.
For example, if  $m=5$, $q=13$, and $i=2$, then $(x_1,x_2,x_3,x_4,x_5)=(12,0,9,6,3)$ corresponds to the following boxes and circles:
\begin{center}
\scalebox{0.8}{
\begin{tikzpicture}[main/.style = {draw, circle, very thick}] 
\draw [very thick] (0,0) rectangle (4,1); \draw [very thick] (5,0) rectangle (9,1); \draw [very thick] (10,0) rectangle (14,1);
\draw [very thick] (0,1.5) rectangle (4,2.5); \draw [very thick] (5,1.5) rectangle (9,2.5); \draw [very thick] (10,1.5) rectangle (14,2.5);
\node[main](0) at (0.5,2) {{\large $2$}}; \node[main](1) at (1.5,2) {{\color{white}\large $0$}};
\node[main](2) at (5.5,2) {{\color{white}\large $0$}}; \node[main](3) at (6.5,2) {{\large $5$}};  \node[main](4) at (7.5,2) {{\color{white}\large $0$}};
\node[main](5) at (10.5,2) {{\color{white}\large $0$}}; \node[main](6) at (11.5,2) {{\large $4$}};
\node[main](12) at (1.5,0.5) {{\large $1$}};
\node[main](11) at (5.5,0.5) {{\color{white}\large $0$}}; \node[main](10) at (6.5,0.5) {{\color{white}\large $0$}}; \node[main](9) at (7.5,0.5) {{\large $3$}};
\node[main](8) at (10.5,0.5) {{\color{white}\large $0$}}; \node[main](7) at (11.5,0.5) {{\color{white}\large $0$}};
\end{tikzpicture}
}
\end{center}
\item The upper left box does not contain any number since $x_s\neq 1$ for any $s\in[m]$.
Each of the numbers $i+1,\dots,m$ cannot be placed into the lower left box.
Indeed, if labeled circles are placed in the lower left box and if the label on the leftmost circle is $s>i$, then we have $x_i-x_s=0-(-1)=1$, which is a contradiction.
\end{itemize}

There are $q-2m+2$ boxes that can contain the numbers $1,\dots,i-1$ except the upper left box, and there are $q-2m+1$ boxes that can contain the numbers $i+1,\dots,m$ except the upper left and lower left boxes.
Therefore, we have that
\begin{align*}
\left|M(\mathcal{A}^{(1)}_q)\right|=(q-2m+1)^{m-i}(q-2m+2)^{i-1}.
\end{align*}

\subsection{The proof of Theorem \ref{chara-quasi-del-by-xi=0} when $q$ is even}\label{sec-quasi-poly-rest-xi=0-even}
In this subsection, we assume that $q$ is even.
We count the number of elements $(x_1,\dots,x_m)\in M(\mathcal{A}^{(1)}_q)$ by dividing the cases by whether there exists an index $k$ such that $\displaystyle x_k=\frac{q}{2}$ or not.

(i)\ Suppose that there is no index $k\in[m]$ such that $\displaystyle x_k=\frac{q}{2}$.
Prepare $q-2m+2$ boxes, $\displaystyle \frac{q}{2}-(m-1)$ on the upper side and $\displaystyle \frac{q}{2}-(m-1)$ on the lower side since $m-1$ numbers make $2(m-1)$ circles in the boxes, and one unlabeled circle corresponding to the element $\displaystyle \frac{q}{2}\in\mathbb{Z}_q$ on the right side of the boxes.
Place each of the numbers $1,\dots,i-1,i+1,\dots,m$ into one of the $q-2m+2$ boxes, and create circles.
When creating the circles, the circle at the left edge of the upper left box is labeled with $i$.
Similarly to the case when $q$ is odd, the upper left box does not contain any number and each of the numbers $i+1,\dots,m$ cannot be placed into the lower left box.

There are $q-2m+1$ boxes that can contain the numbers $1,\dots,i-1$ except the upper left box, and there are $q-2m$ boxes that can contain the numbers $i+1,\dots,m$ except the upper left and lower left boxes. 
Therefore, in this case, the desired number is
\begin{align*}
(q-2m)^{m-i}(q-2m+1)^{i-1}.
\end{align*}

(i\hspace{-0.5mm}i)\ Suppose that there exists an index $k\in[m]$ such that $\displaystyle x_k=\frac{q}{2}$.
Create boxes and circles according to the following rules and take the corresponding element $(x_1,\dots,x_m)\in M(\mathcal{A}^{(1)}_q)$.
\begin{itemize}
\item Prepare $q-2m+4$ boxes, $\displaystyle \frac{q}{2}-(m-2)$ on the upper side and $\displaystyle \frac{q}{2}-(m-2)$ on the lower side since $m-2$ numbers make $2(m-2)$ circles in the boxes, and one circle labeled with $k$ on the right side of the boxes.
Place each of the numbers that are neither $i$ nor $k$ into one of $q-2m+4$ boxes.
The circle at the left edge of the upper left box is labeled with $i$.
\end{itemize}
For example, if $m=5$, $q=12$, $i=4$, and $k=2$ then $(x_1,x_2,x_3,x_4,x_5)=(5,6,11,0,3)$ corresponds to the following boxes and circles:
\begin{center}
\scalebox{0.8}{
\begin{tikzpicture}[main/.style = {draw, circle, very thick}] 
\draw [very thick] (0,0) rectangle (4,1); \draw [very thick] (5,0) rectangle (9,1); \draw [very thick] (10,0) rectangle (14,1);
\draw [very thick] (0,1.5) rectangle (4,2.5); \draw [very thick] (5,1.5) rectangle (9,2.5); \draw [very thick] (10,1.5) rectangle (14,2.5);
\node[main](6) at (15,1.25) {{\large $2$}};
\node[main](0) at (0.5,2) {{\large $4$}}; \node[main](1) at (1.5,2) {{\color{white}\large $0$}};
\node[main](2) at (5.5,2) {{\color{white}\large $0$}}; \node[main](3) at (6.5,2) {{\large $5$}};
\node[main](4) at (10.5,2) {{\color{white}\large $0$}}; \node[main](5) at (11.5,2) {{\large $1$}};
\node[main](11) at (1.5,0.5) {\large $3$}; \node[main](10) at (5.5,0.5) {{\color{white}\large $0$}}; \node[main](9) at (6.5,0.5) {{\color{white}\large $0$}};
\node[main](8) at (10.5,0.5) {{\color{white}\large $0$}}; \node[main](7) at (11.5,0.5) {{\color{white}\large $0$}};
\end{tikzpicture}
}
\end{center}
\begin{itemize}
\item The upper left box does not contain any number and each of the numbers $i+1,\dots,m$ cannot be placed into the lower left box.
\item For the same reasons discussed in Section \ref{sec-count-method-even}, the lower right box does not contain any numbers, since $x_k\neq -x_s+1$ for any $s\in [m]$ with $k\neq s$.
The upper right box does not contain any numbers larger than $k$, since $x_k\neq x_t+1$ for any $t\in [m]$ with $k<t$.
\end{itemize}

(i\hspace{-0.5mm}i-1)\ Suppose that $1\leq k\leq i-1$.
There are $q-2m+2$ boxes that can contain the numbers $1,\dots,k-1$ except the upper left and the lower right boxes, and $q-2m+1$ boxes that can contain the numbers $k+1,\dots,i-1$ except the upper left, the upper right, and the lower right boxes.
For the remaining numbers, there are $q-2m$ boxes that can contain the numbers $i+1,\dots,m$, except the upper left, the upper right, lower left, and the lower right boxes.
Therefore, in this case, the desired number is
\begin{align*}
(q-2m)^{m-i}(q-2m+1)^{i-k-1}(q-2m+2)^{k-1}.
\end{align*}

(i\hspace{-0.5mm}i-2)\ Suppose that $i+1\leq k\leq m$.
There are $q-2m+2$ boxes that can contain the numbers $1,\dots,i-1$, except the upper left and the lower right boxes.
For the remaining numbers, there are $q-2m+1$ boxes that can contain the numbers $i+1,\dots,k-1$ except the upper left, the lower left, and the lower right boxes, and $q-2m$ boxes that can contain the numbers $k+1,\dots,m$ except the upper left, the upper right, lower left, and the lower right boxes.
Therefore, in this case, the desired number is
\begin{align*}
(q-2m)^{m-k}(q-2m+1)^{k-i-1}(q-2m+2)^{i-1}.
\end{align*}

From the discussion above, by applying Lemma~\ref{hodai}, we have that
\begin{align*}
\left|M(\mathcal{A}^{(1)}_q)\right|
=&\,T^{m-i}(T+1)^{i-1}+\sum_{k=1}^{i-1}T^{m-i}(T+1)^{i-k-1}(T+2)^{k-1}+\sum_{k=i+1}^{m}T^{m-k}(T+1)^{k-i-1}(T+2)^{i-1}\\
=&\,T^{m-i}(T+1)^{i-1}+T^{m-i}\left((T+2)^{i-1}-(T+1)^{i-1}\right)+(T+2)^{i-1}\left((T+1)^{m-i}-T^{m-i}\right)\\
=&\,(T+1)^{m-i}(T+2)^{i-1}, 
\end{align*}
where $T:=q-2m$.

\section{Characteristic quasi-polynomial of restriction on $\{x_i=1\}$}\label{sec-quasi-poly-restxi=1}

Let $\mathcal{A}^{(2)}=\mathcal{B}_m^{\{x_i=1\}}$ for $i\in [m]$, and we prove the second equality of Theorem \ref{thm:main} as follows.
\begin{Them}\label{chara-quasi-del-by-xi=1}
We have
\begin{align*}
\left|M(\mathcal{A}^{(2)}_q)\right|=(q-2m)^{i-1}(q-2m+1)^{m-i}
\end{align*}
for any $q\in\mathbb{Z}$ with $q\gg 0$.
\end{Them}

The complement $M(\mathcal{A}^{(2)}_q)$ is the set of $(x_1,\dots,x_m)\in\mathbb{Z}_q^m$ satisfying the following conditions:
\begin{align*}
&x_i=1,\\
&x_s\neq 0,\ x_s\neq 1\ (s\in [m],\ s\neq i),\\
&x_s\neq x_t\ (s,t\in [m],\ s\neq t),\\
&x_s\neq x_t+1\ (s,t\in [m],\ s<t),\\
&x_s\neq -x_t,\ x_s\neq -x_t+1\ (s,t\in [m],\ s\neq t).
\end{align*}
We fix the condition $x_i=1$ and count the number of elements in $(x_1,\dots,x_m)\in M(\mathcal{A}^{(2)}_q)$ using a modified version of the counting method described in Section \ref{sec-count-method}.
The condition $x_i=1$ means that the second circle from the left in the upper left box is the circle labeled with $i$.
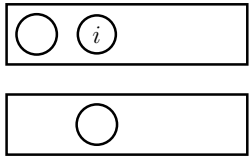
\begin{center}
\scalebox{0.8}{
\begin{tikzpicture}[main/.style = {draw, circle, very thick}] 
\draw [very thick] (0,0) rectangle (4,1); \draw [very thick] (0,1.5) rectangle (4,2.5);
\node[main](0) at (0.5,2) {{\color{white}\large $0$}}; \node[main](1) at (1.5,2) {{\large $i$}}; \node[main](2) at (1.5,0.5) {{\color{white}\large $0$}};
\end{tikzpicture}
}
\end{center}

The remaining parts of this section are devoted to giving the proof of Theorem~\ref{chara-quasi-del-by-xi=1}. 

\subsection{The proof of Theorem \ref{chara-quasi-del-by-xi=1} when $q$ is odd}\label{sec-quasi-poly-rest-xi=1-odd}
We assume that $q$ is odd.
Making a modification to the procedure in Section \ref{sec-count-method-odd}, we create boxes and circles and take the corresponding element $(x_1,\dots,x_m)\in M(\mathcal{A}^{(2)}_q)$.
\begin{itemize}
\item Prepare $q-2m+1$ boxes side by side, $\displaystyle \frac{q+1}{2}-m$ on the upper side and $\displaystyle \frac{q+1}{2}-m$ on the lower side.
Place each of the numbers $1,\dots,i-1,i+1,\dots,m$ into one of $q-2m+1$ boxes.
Create circles by the same procedure described in Section \ref{sec-count-method-odd} and place the circle labeled with $i$ next to the leftmost circle in the upper left box.
In addition, place the unlabeled circle on the opposite side of the circle labeled with $i$.
For example, if  $m=5$, $q=15$, and $i=3$, then $(x_1,x_2,x_3,x_4,x_5)=(12,10,1,7,2)$ corresponds to the following boxes and circles:
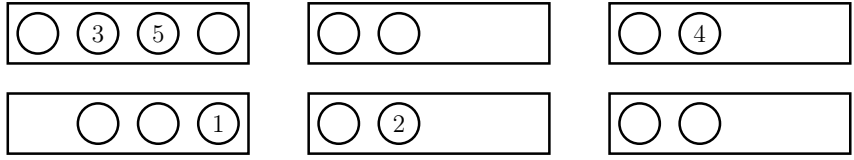
\begin{center}
\scalebox{0.8}{
\begin{tikzpicture}[main/.style = {draw, circle, very thick}] 
\draw [very thick] (0,0) rectangle (4,1); \draw [very thick] (5,0) rectangle (9,1); \draw [very thick] (10,0) rectangle (14,1);
\draw [very thick] (0,1.5) rectangle (4,2.5); \draw [very thick] (5,1.5) rectangle (9,2.5); \draw [very thick] (10,1.5) rectangle (14,2.5);
\node[main](0) at (0.5,2) {{\color{white}\large $0$}}; \node[main](1) at (1.5,2) {{\large $3$}}; \node[main](2) at (2.5,2) {{\large $5$}}; \node[main](3) at (3.5,2) {{\color{white}\large $0$}};
\node[main](4) at (5.5,2) {{\color{white}\large $0$}}; \node[main](5) at (6.5,2) {{\color{white}\large $0$}};
\node[main](6) at (10.5,2) {{\color{white}\large $0$}}; \node[main](7) at (11.5,2) {{\large $4$}};
\node[main](14) at (1.5,0.5) {{\color{white}\large $0$}}; \node[main](13) at (2.5,0.5) {{\color{white}\large $0$}}; \node[main](12) at (3.5,0.5) {{\large $1$}};
\node[main](11) at (5.5,0.5) {{\color{white}\large $0$}}; \node[main](10) at (6.5,0.5) {{\large $2$}};
\node[main](9) at (10.5,0.5) {{\color{white}\large $0$}}; \node[main](8) at (11.5,0.5) {{\color{white}\large $0$}};
\end{tikzpicture}
}
\end{center}
\item Each of the numbers $1,\dots,i-1$ cannot be placed into the upper left box and each of the numbers $i+1,\dots,m$ can be placed anywhere.
Indeed, a circle with a label smaller than $i$ does not precede the circle labeled with $i$ in a clockwise direction since $i$ is the second.
When the number $s$ $(s>i)$ is in the upper left box, the circle labeled with $s$ is placed in after the circle labeled with $i$, so no contradiction occurs.
\end{itemize}

There are $q-2m$ boxes that can contain the numbers $1,\dots,i-1$ except the upper left box, and there are $q-2m+1$ boxes that can contain the numbers $i+1,\dots,m$.
Therefore, we have that
\begin{align*}
\left|M(\mathcal{A}^{(2)}_q)\right|=(q-2m)^{i-1}(q-2m+1)^{m-i}.
\end{align*}

\subsection{The proof of Theorem \ref{chara-quasi-del-by-xi=1} when $q$ is even}\label{sec-quasi-poly-rest-xi=1-even}
In this subsection, we assume that $q$ is even.
We count the number of elements $(x_1,\dots,x_m)\in M(\mathcal{A}^{(2)}_q)$ by dividing the cases by whether there exists an index $k$ such that $\displaystyle x_k=\frac{q}{2}$ or not.

(i)\ Suppose that there is no index $k\in[m]$ such that $\displaystyle x_k=\frac{q}{2}$.
Prepare $q-2m$ boxes, $\displaystyle \frac{q}{2}-m$ on the upper side and $\displaystyle \frac{q}{2}-m$ on the lower side, and one unlabeled circle corresponding to the element $\displaystyle \frac{q}{2}\in\mathbb{Z}_q$ on the right side of the boxes.
Place each of the numbers $1,\dots,i-1,i+1,\dots,m$ into one of the $q-2m$ boxes, and create circles.
When creating the circles, place the circle labeled with $i$ next to the leftmost circle in the upper left box and the unlabeled circle on the opposite side of the circle labeled with $i$.
Similarly to the case when $q$ is odd, each of the numbers $1,\dots,i-1$ cannot be placed into the upper left box and each of the numbers $i+1,\dots,m$ can be placed anywhere.

There are $q-2m-1$ boxes that can contain the numbers $1,\dots,i-1$ except the upper left box, and each of the numbers $i+1,\dots,m$ can be contained anywhere $q-2m$ boxes. 
Therefore, in this case, the desired number is
\begin{align*}
(q-2m-1)^{i-1}(q-2m)^{m-i}.
\end{align*}

(i\hspace{-0.5mm}i)\ Suppose that there exists an index $k\in[m]$ such that $\displaystyle x_k=\frac{q}{2}$.
Create boxes and circles according to the following rules and take the corresponding element $(x_1,\dots,x_m)\in M(\mathcal{A}^{(2)}_q)$.
\begin{itemize}
\item Prepare $q-2m+2$ boxes, $\displaystyle \frac{q}{2}-(m-1)$ on the upper side and $\displaystyle \frac{q}{2}-(m-1)$ on the lower side, and one circle labeled with $k$ on the right side of the boxes.
Place each of the numbers that are neither $i$ nor $k$ into one of $q-2m+2$ boxes.
\item Place the circle labeled with $i$ next to the leftmost circle in the upper left box and the unlabeled circle on the opposite side of the circle labeled with $i$.
\end{itemize}
For example, if $m=5$, $q=16$, $i=3$, and $k=1$ then $(x_1,x_2,x_3,x_4,x_5)=(8,10,1,3,11)$ corresponds to the following boxes and circles:
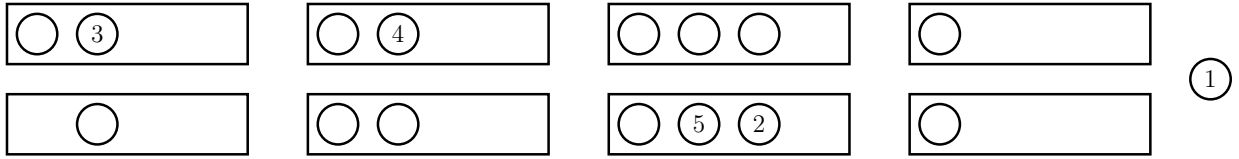
\begin{center}
\scalebox{0.8}{
\begin{tikzpicture}[main/.style = {draw, circle, very thick}] 
\draw [very thick] (0,0) rectangle (4,1); \draw [very thick] (5,0) rectangle (9,1); \draw [very thick] (10,0) rectangle (14,1); \draw [very thick] (15,0) rectangle (19,1);
\draw [very thick] (0,1.5) rectangle (4,2.5); \draw [very thick] (5,1.5) rectangle (9,2.5); \draw [very thick] (10,1.5) rectangle (14,2.5); \draw [very thick] (15,1.5) rectangle (19,2.5);
\node[main](8) at (20,1.25) {{\large $1$}};
\node[main](0) at (0.5,2) {{\color{white}\large $0$}}; \node[main](1) at (1.5,2) {{\large $3$}};
\node[main](2) at (5.5,2) {{\color{white}\large $0$}}; \node[main](3) at (6.5,2) {{\large $4$}};
\node[main](4) at (10.5,2) {{\color{white}\large $0$}}; \node[main](5) at (11.5,2) {{\color{white}\large $0$}}; \node[main](6) at (12.5,2) {{\color{white}\large $0$}};
\node[main](7) at (15.5,2) {{\color{white}\large $0$}};
\node[main](15) at (1.5,0.5) {{\color{white}\large $0$}};
\node[main](14) at (5.5,0.5) {{\color{white}\large $0$}}; \node[main](13) at (6.5,0.5) {{\color{white}\large $0$}};
\node[main](12) at (10.5,0.5) {{\color{white}\large $0$}}; \node[main](11) at (11.5,0.5) {{\large $5$}}; \node[main](10) at (12.5,0.5) {{\large $2$}};
\node[main](9) at (15.5,0.5) {{\color{white}\large $0$}};
\end{tikzpicture}
}
\end{center}
\begin{itemize}
\item Each of the numbers $1,\dots,i-1$ cannot be placed into the upper left box.
On the other hand, each of the numbers $i+1,\dots,m$ can be placed into the upper left box.
\item For the same reasons discussed in Section \ref{sec-count-method-even}, the lower right box does not contain any numbers, since $x_k\neq -x_s+1$ for any $s\in [m]$ with $k\neq s$.
The upper right box does not contain any numbers larger than $k$, since $x_k\neq x_t+1$ for any $t\in [m]$ with $k<t$.
\end{itemize}

(i\hspace{-0.5mm}i-1)\ Suppose that $1\leq k\leq i-1$.
There are $q-2m$ boxes that can contain the numbers $1,\dots,k-1$ except the upper left and the lower right boxes, and $q-2m-1$ boxes that can contain the numbers $k+1,\dots,i-1$ except the upper left, the upper right, and the lower right boxes.
For the remaining numbers, there are $q-2m$ boxes that can contain the numbers $i+1,\dots,m$, except the upper right and the lower right boxes.
Therefore, in this case, the desired number is
\begin{align*}
(q-2m-1)^{i-k-1}(q-2m)^{m-i+k-1}.
\end{align*}

(i\hspace{-0.5mm}i-2)\ Suppose that $i+1\leq k\leq m$.
There are $q-2m$ boxes that can contain the numbers $1,\dots,i-1$, except the upper left and the lower right boxes.
For the remaining numbers, there are $q-2m+1$ boxes that can contain the numbers $i+1,\dots,k-1$ except the lower right box, and $q-2m$ boxes that can contain the numbers $k+1,\dots,m$ except the upper right and the lower right boxes.
Therefore, in this case, the desired number is
\begin{align*}
(q-2m)^{m-k+i-1}(q-2m+1)^{k-i-1}.
\end{align*}

From the discussion above, by applying Lemma~\ref{hodai}, we have that
\begin{align*}
\left|M(\mathcal{A}^{(2)}_q)\right|
=&\,(T-1)^{i-1}T^{m-i}+\sum_{k=1}^{i-1}(T-1)^{i-k-1}T^{m-i+k-1}+\sum_{k=i+1}^{m}T^{m-k+i-1}(T+1)^{k-i-1}\\
=&\,(T-1)^{i-1}T^{m-i}+T^{m-i}\left(T^{i-1}-(T-1)^{i-1}\right)+T^{i-1}\left((T+1)^{m-i}-T^{m-i}\right)\\
=&\,T^{i-1}(T+1)^{m-i}, 
\end{align*}
where $T:=q-2m$.

\section{Characteristic quasi-polynomial of restriction on $\{x_i-x_j=0\}$}\label{sec-quasi-poly-restxi=xj}

Let $\mathcal{A}^{(3)}=\mathcal{B}_m^{\{x_i-x_j=0\}}$ for $1\leq i<j \leq m$, and we prove the following theorem.
\begin{Them}\label{chara-quasi-del-by-xi=xj}
We have
\begin{align*}
\left|M(\mathcal{A}^{(3)}_q)\right|=
\begin{cases}
(q-2m+1)^{j-i-1}(q-2m+2)^{m-j}((q-2m+2)^i-(q-2m+1)^{i-1}) \;\;\;&(q\ \text{is odd}), \\
(q-2m+1)^{j-i-1}(q-2m+2)^{i-1}((q-2m+2)^{m-j+1}-(q-2m+1)^{m-j}) &(q\ \text{is even}).
\end{cases}
\end{align*}
\end{Them}
The complement $M(\mathcal{A}^{(3)}_q)$ is the set of $(x_1,\dots,x_m)\in\mathbb{Z}_q^m$ satisfying the following conditions:
\begin{align*}
&x_i=x_j,\\ 
&x_s\neq 0,\ x_s\neq 1\ (s\in [m]),\\
&x_s\neq x_t\ (s,t\in [m],\ s\neq t,\ \{s,t\}\neq \{i,j\}),\\
&x_s\neq x_t+1\ (s,t\in [m],\ s<t,\ (s,t)\neq (i,j)),\\
&x_s\neq -x_t,\ x_s\neq -x_t+1\ (s,t\in [m],\ s\neq t).
\end{align*}
To prove Theorem \ref{chara-quasi-del-by-xi=xj}, we fix the condition $x_i=x_j$ and count the number of elements in $(x_1,\dots,x_m)\in M(\mathcal{A}^{(3)}_q)$ using a modified version of the counting method described in Section \ref{sec-count-method}.
The condition $x_i=x_j$ means that a circle has label $i$ and label $j$ at the same time.
Thus we consider the numbers $i$ and $j$ as the pair $(i,j)$.
In the examples below, the pair $(i,j)$ is denoted as the smaller number $i$.

\subsection{The proof of Theorem \ref{chara-quasi-del-by-xi=xj} when $q$ is odd}\label{sec-quasi-poly-restxi=xj-odd}
We assume that $q$ is odd.
Making a modification to the procedure in Section \ref{sec-count-method-odd}, we create boxes and circles, and take the corresponding element $(x_1,\dots,x_m)\in M(\mathcal{A}^{(3)}_q)$, as follows:
\begin{itemize}
\item Prepare $q-2m+3$ boxes side by side, $\displaystyle \frac{q+1}{2}-(m-1)$ on the upper side and $\displaystyle \frac{q+1}{2}-(m-1)$ on the lower side.
Place each of the numbers $1,\dots,j-1,j+1,\dots,m$ in one of $q-2m+3$ boxes.
The upper left box does not contain a number since $x_s\neq 1$ for any $s\in [m]$.
\end{itemize}
For example, if $m=5$, $q=15$, $i=2$, and $j=4$, then $(x_1,x_2,x_3,x_4,x_5)=(8,2,11,2,3)$ corresponds to the following boxes and circles:
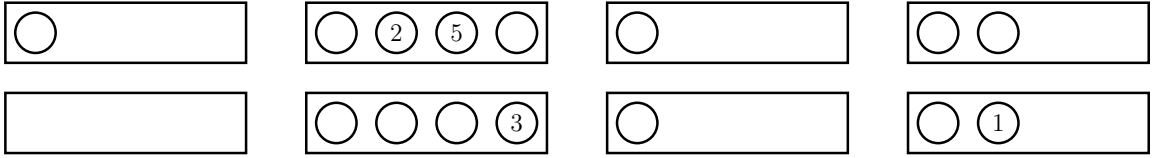
\begin{center}
\scalebox{0.8}{
\begin{tikzpicture}[main/.style = {draw, circle, very thick}] 
\draw [very thick] (0,0) rectangle (4,1); \draw [very thick] (5,0) rectangle (9,1); \draw [very thick] (10,0) rectangle (14,1); \draw [very thick] (15,0) rectangle (19,1);
\draw [very thick] (0,1.5) rectangle (4,2.5); \draw [very thick] (5,1.5) rectangle (9,2.5); \draw [very thick] (10,1.5) rectangle (14,2.5); \draw [very thick] (15,1.5) rectangle (19,2.5);
\node[main](1) at (0.5,2) {{\color{white}\large $0$}};
\node[main](2) at (5.5,2) {{\color{white}\large $0$}}; \node[main](3) at (6.5,2) {{\large $2$}}; \node[main](4) at (7.5,2) {{\large $5$}}; \node[main](5) at (8.5,2) {{\color{white}\large $0$}};
\node[main](6) at (10.5,2) {{\color{white}\large $0$}};
\node[main](7) at (15.5,2) {{\color{white}\large $0$}}; \node[main](8) at (16.5,2) {{\color{white}\large $0$}};
\node[main](9) at (5.5,0.5) {{\color{white}\large $0$}}; \node[main](10) at (6.5,0.5) {{\color{white}\large $0$}}; \node[main](11) at (7.5,0.5) {{\color{white}\large $0$}}; \node[main](12) at (8.5,0.5) {{\large $3$}};
\node[main](13) at (10.5,0.5) {{\color{white}\large $0$}};
\node[main](14) at (15.5,0.5) {{\color{white}\large $0$}}; \node[main](15) at (16.5,0.5) {{\large $1$}};
\end{tikzpicture}
}
\end{center}
\begin{itemize}
\item The box containing the pair $(i,j)$ cannot contain the numbers $i+1,\dots,j-1$.
Indeed, if the clockwise next circle of the circle labeled with $(i,j)$ is the circle labeled with $s$ $(i+1\leq s\leq j-1)$, then $x_s=x_j+1$ and $s<j$, which is a contradiction.
We note that $x_s=x_i+1$ also holds, but it is not a contradiction.
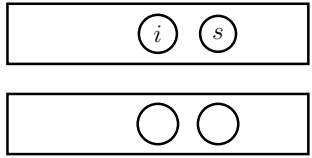
\begin{center}
\scalebox{0.8}{
\begin{tikzpicture}[main/.style = {draw, circle, very thick}] 
\draw [very thick] (0,0) rectangle (5,1); \draw [very thick] (0,1.5) rectangle (5,2.5);
\node[main](1) at (2.5,0.5) {{\color{white}\large $0$}}; \node[main](2) at (3.5,0.5) {{\color{white}\large $0$}}; \node[main](3) at (2.5,2) {{\large $i$}}; \node[main](4) at (3.5,2) {{\large $s$}};
\end{tikzpicture}
}
\end{center}
\item The circle at the right end of the lower right box is not the circle labeled with $(i,j)$.
Indeed, if so, then the clockwise next circle on the opposite side of the circle labeled with $(i,j)$ is itself.
This implies $x_j=-x_i+1$, which is a contradiction.
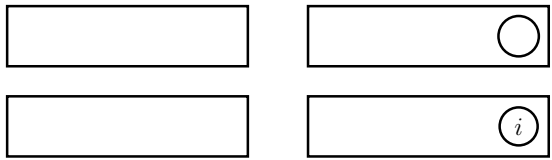
\begin{center}
\scalebox{0.8}{
\begin{tikzpicture}[main/.style = {draw, circle, very thick}] 
\draw [very thick] (0,0) rectangle (4,1); \draw [very thick] (5,0) rectangle (9,1); \draw [very thick] (0,1.5) rectangle (4,2.5); \draw [very thick] (5,1.5) rectangle (9,2.5);
\node[main](1) at (8.5,0.5) {{\large $i$}}; \node[main](2) at (8.5,2) {{\color{white}\large $0$}};
\end{tikzpicture}
}
\end{center}
\end{itemize}

We count the number of elements $(x_1,\dots,x_m)\in M(\mathcal{A}^{(3)}_q)$ by dividing the cases according to whether the lower right box contains the numbers $1,\dots,i-1$.

(i) Suppose that there is no number $1,\dots,i-1$ in the lower right box.
There are $q-2m+1$ boxes that can contain the numbers $1,\dots,i-1$, except the upper left and the lower right boxes.
In this case, the pair $(i,j)$ cannot be contained in the lower right box.
Thus there are $q-2m+1$ boxes for the pair $(i,j)$, except the upper left and the lower right boxes.
For the remaining numbers, the numbers $i+1,\dots,j-1$ cannot be placed into the same box as $(i,j)$ and the upper left box, so $q-2m+1$ boxes can be chosen.
The numbers $j+1,\dots,m$ can be placed anywhere except the upper left box, so $q-2m+2$ boxes can be chosen.
Therefore the desired number in this case is 
\begin{align*}
(q-2m+1)^{j-1}(q-2m+2)^{m-j}.
\end{align*}

(i\hspace{-0.5mm}i)\ Suppose one of the numbers $1,\dots,i-1$ is in the lower right box.
The total number to place $1,\dots,i-1$ in this case is $(q-2m+2)^{i-1}-(q-2m+1)^{i-1}$.
Indeed, the number to place $1,\dots,i-1$ in $q-2m+2$ boxes except the upper left box is $(q-2m+2)^{i-1}$, and from (i), the number of cases where there is no numbers $1,\dots,i-1$ in the upper left and lower right boxes is $(q-2m+1)^{i-1}$.
Next, since the circle labeled with $(i,j)$ is never placed on the right end of the lower right box, the pair $(i,j)$ can be placed into the lower right box.
Thus there are $q-2m+2$ boxes to place the pair $(i,j)$, except the upper left box.
For the remaining numbers, as in (i), the numbers $i+1,\dots,j-1$ cannot be placed in the same box as $(i,j)$ and the upper left box, so $q-2m+1$ boxes can be chosen.
The numbers $j+1,\dots,m$ can be placed anywhere except the upper left box, so $q-2m+2$ boxes can be chosen.
Therefore the desired number in this case is
\begin{align*}
\left((q-2m+2)^{i-1}-(q-2m+1)^{i-1}\right)(q-2m+1)^{j-i-1}(q-2m+2)^{m-j+1}.
\end{align*}

From the discussion above, we have that
\begin{align*}
\left|M(\mathcal{A}^{(3)}_q)\right|&=U^{j-1}(U+1)^{m-j}+\left((U+1)^{i-1}-U^{i-1}\right)U^{j-i-1}(U+1)^{m-j+1}\\
&=U^{j-1}(U+1)^{m-j}+U^{j-i-1}(U+1)^{m+i-j}-U^{j-2}(U+1)^{m-j+1}\\
&=U^{j-i-1}(U+1)^{m+i-j}+U^{j-2}(U+1)^{m-j}\left(U-(U+1)\right)\\
&=U^{j-i-1}(U+1)^{m-j}\left((U+1)^i-U^{i-1}\right),
\end{align*}
where $U:=q-2m+1$.

\subsection{The proof of Theorem \ref{chara-quasi-del-by-xi=xj} when $q$ is even}\label{sec-quasi-poly-restxi=xj-even}
In this subsection, we assume that $q$ is even.
We count the number of elements $(x_1,\dots,x_m)\in M(\mathcal{A}^{(3)}_q)$ by dividing the cases by whether there exists an index $k$ such that $\displaystyle x_k=\frac{q}{2}$ or not.

(i)\ Suppose that there is no index $k\in[m]$ such that $\displaystyle x_k=\frac{q}{2}$.
Prepare $q-2m+2$ boxes, $\displaystyle \frac{q}{2}-(m-1)$ on the upper side and $\displaystyle \frac{q}{2}-(m-1)$ on the lower side, and one unlabeled circle corresponding to the element $\displaystyle \frac{q}{2}\in\mathbb{Z}_q$ on the right side of the boxes.
Put each of the numbers $1,\dots,i-1,i+1,\dots,j-1,j+1,\dots,m$ and the pair $(i,j)$ into one of the $q-2m+1$ boxes, except the upper left box.
For example, if $m=5$, $q=16$, $i=2$, and $j=4$, then $(x_1,x_2,x_3,x_4,x_5)=(3,9,6,9,4)$ corresponds to the following boxes and circles:
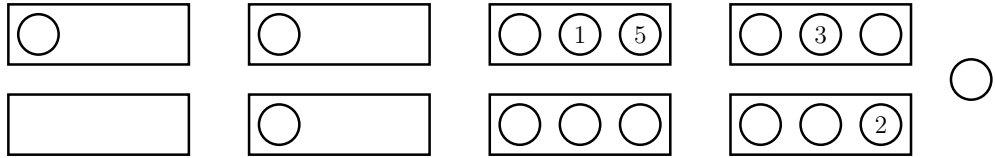
\begin{center}
\scalebox{0.8}{
\begin{tikzpicture}[main/.style = {draw, circle, very thick}] 
\draw [very thick] (0,0) rectangle (3,1); \draw [very thick] (4,0) rectangle (7,1); \draw [very thick] (8,0) rectangle (11,1); \draw [very thick] (12,0) rectangle (15,1);
\draw [very thick] (0,1.5) rectangle (3,2.5); \draw [very thick] (4,1.5) rectangle (7,2.5); \draw [very thick] (8,1.5) rectangle (11,2.5); \draw [very thick] (12,1.5) rectangle (15,2.5);
\node[main](8) at (16,1.25) {{\color{white}\large $0$}};
\node[main](0) at (0.5,2) {{\color{white}\large $0$}};
\node[main](1) at (4.5,2) {{\color{white}\large $0$}};
\node[main](2) at (8.5,2) {{\color{white}\large $0$}}; \node[main](3) at (9.5,2) {{\large $1$}}; \node[main](4) at (10.5,2) {{\large $5$}};
\node[main](5) at (12.5,2) {{\color{white}\large $0$}}; \node[main](6) at (13.5,2) {{\large $3$}}; \node[main](7) at (14.5,2) {{\color{white}\large $0$}};
\node[main](15) at (4.5,0.5) {{\color{white}\large $0$}};
\node[main](14) at (8.5,0.5) {{\color{white}\large $0$}}; \node[main](13) at (9.5,0.5) {{\color{white}\large $0$}}; \node[main](12) at (10.5,0.5) {{\color{white}\large $0$}};
\node[main](11) at (12.5,0.5) {{\color{white}\large $0$}}; \node[main](10) at (13.5,0.5) {{\color{white}\large $0$}}; \node[main](10) at (14.5,0.5) {{\large $2$}};
\end{tikzpicture}
}
\end{center}
As in the case when $q$ is odd, each of the numbers $i+1,\dots,j-1$ cannot be placed in the same box as the pair $(i,j)$.
On the other hand, unlike the case when $q$ is odd, the circle labeled with $(i,j)$ can be placed on the right end of the lower right box.
Indeed, if so, then the clockwise next circle on the opposite side of the circle labeled with $(i,j)$ is the unlabeled circle.
We create $(x_1,\dots,x_m)\in M(\mathcal{A}^{(3)}_q)$ by the same procedure described in Section \ref{sec-count-method}.

Each of $1,\dots,i-1$ and $(i,j)$ can be placed into one of $q-2m+1$ boxes, except the upper left box.
Then each of the numbers $i+1,\dots,j-1$ can be placed into one of $q-2m$ boxes except the upper left box and the box containing $(i,j)$, and each of the numbers $j+1,\dots,m$ can be placed into one of $q-2m+1$ boxes except only the upper left box.
Therefore, in this case, the desired number is
\begin{align*}
(q-2m)^{j-i-1}(q-2m+1)^{m+i-j}.
\end{align*}

(i\hspace{-0.5mm}i)\ Suppose that there exists an index $k\in[m]$ such that $\displaystyle x_k=\frac{q}{2}$.
Create boxes and circles according to the following rules, and take the corresponding element $(x_1,\dots,x_m)\in M(\mathcal{A}^{(3)}_q)$.
\begin{itemize}
\item Prepare $q-2m+4$ boxes, $\displaystyle \frac{q}{2}-(m-2)$ on the upper side and $\displaystyle \frac{q}{2}-(m-2)$ on the lower side, and one circle labeled with $k$ on the right side of the boxes.
Place each of $(i,j)$ and the numbers that is not $i$, $j$, or $k$ into one of $q-2m+3$ boxes, except the upper left box.
\item As in the case when $q$ is odd, each of the numbers $i+1,\dots,j-1$ cannot be placed in the same box as the pair $(i,j)$.
\item For the same reasons discussed in Section \ref{sec-count-method-even}, the lower right box does not contain any numbers, since $x_k\neq -x_s+1$ for any $s\in [m]$ with $k\neq s$.
The upper right box does not contain any numbers larger than $k$, since $x_k\neq x_t+1$ for any $t \in [m]$ with $k<t$.
\end{itemize}
For example, if $m=5$, $q=16$, $i=2$, $j=4$, and $k=5$, then $(x_1,x_2,x_3,x_4,x_5)=(7,15,12,15,8)$ corresponds to the following boxes and circles:
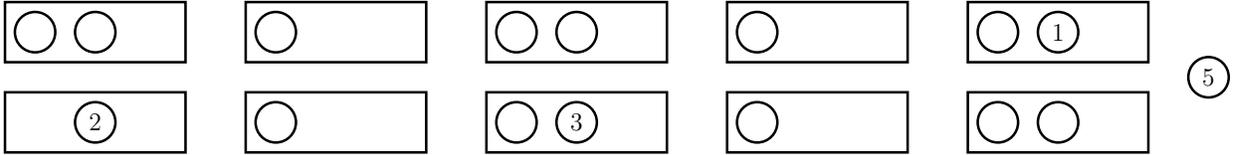
\begin{center}
\scalebox{0.8}{
\begin{tikzpicture}[main/.style = {draw, circle, very thick}] 
\draw [very thick] (0,0) rectangle (3,1); \draw [very thick] (4,0) rectangle (7,1); \draw [very thick] (8,0) rectangle (11,1); \draw [very thick] (12,0) rectangle (15,1); \draw [very thick] (16,0) rectangle (19,1); 
\draw [very thick] (0,1.5) rectangle (3,2.5); \draw [very thick] (4,1.5) rectangle (7,2.5); \draw [very thick] (8,1.5) rectangle (11,2.5); \draw [very thick] (12,1.5) rectangle (15,2.5); \draw [very thick] (16,1.5) rectangle (19,2.5);
\node[main](8) at (20,1.25) {{\large $5$}};
\node[main](0) at (0.5,2) {{\color{white}\large $0$}}; \node[main](1) at (1.5,2) {{\color{white}\large $0$}};
\node[main](2) at (4.5,2) {{\color{white}\large $0$}};
\node[main](3) at (8.5,2) {{\color{white}\large $0$}}; \node[main](4) at (9.5,2) {{\color{white}\large $0$}};
\node[main](5) at (12.5,2) {{\color{white}\large $0$}};
\node[main](6) at (16.5,2) {{\color{white}\large $0$}}; \node[main](7) at (17.5,2) {{\large $1$}};
\node[main](15) at (1.5,0.5) {{\large $2$}};
\node[main](14) at (4.5,0.5) {{\color{white}\large $0$}};
\node[main](13) at (8.5,0.5) {{\color{white}\large $0$}}; \node[main](12) at (9.5,0.5) {{\large $3$}};
\node[main](11) at (12.5,0.5) {{\color{white}\large $0$}};
\node[main](10) at (16.5,0.5) {{\color{white}\large $0$}}; \node[main](9) at (17.5,0.5) {{\color{white}\large $0$}};
\end{tikzpicture}
}
\end{center}

(i\hspace{-0.5mm}i-1)\ Suppose that $1\leq k\leq i-1$.
Each of the numbers $1,\dots,k-1$ can be placed into one of $q-2m+2$ boxes, except the upper left and the lower right boxes.
Each of the numbers $k+1,\dots,i-1$ and the pair $(i,j)$ can be placed into one of $q-2m+1$ boxes, except the upper left, the upper right, and the lower right boxes.
For the remaining numbers, each of the numbers $i+1,\dots,j-1$ can be placed into one of $q-2m$ boxes, except the upper left, the upper right, and the lower right boxes and the box containing $(i,j)$.
Each of the numbers $j+1,\dots,m$ can be placed into one of $q-2m+1$ boxes, except the upper left, the upper right, and the lower right boxes.
Therefore, in this case, the desired number is
\begin{align*}
(q-2m)^{j-i-1}(q-2m+1)^{m+i-j-k}(q-2m+2)^{k-1}.
\end{align*}

(i\hspace{-0.5mm}i-2)\ Suppose that $k=i$.
Since $\displaystyle x_i=x_j=\frac{q}{2}$, we have $x_i+x_j=0$, which is a contradiction.

(i\hspace{-0.5mm}i-3)\ Suppose that $i+1\leq k\leq j-1$.
Each of the numbers $1,\dots,i-1$ can be placed into one of $q-2m+2$ boxes, except the upper left and the lower right boxes.
The pair $(i,j)$ cannot be placed into the upper right box since $j>k$.
Then $(i,j)$ can be placed into one of $q-2m+1$ boxes, except the upper left, the upper right, and the lower right boxes.
For the remaining numbers, each of the numbers $i+1,\dots,k-1$ can be placed into one of $q-2m+1$ boxes, except the upper left and the lower right boxes and the box containing $(i,j)$.
Each of the numbers $k+1,\dots,j-1$ can be placed into one of $q-2m$ boxes, except the upper left, the upper right, and the lower right boxes and the box containing $(i,j)$.
Finally, each of the numbers $j+1,\dots,m$ can be placed into one of $q-2m+1$ boxes, except the upper left, the upper right, and the lower right boxes.
Therefore, in this case, the desired number is
\begin{align*}
(q-2m)^{j-k-1}(q-2m+1)^{m-i-j+k}(q-2m+2)^{i-1}.
\end{align*}

(i\hspace{-0.5mm}i-4)\ Suppose that $j+1\leq k\leq m$.
Each of the numbers $1,\dots,i-1$ and $(i,j)$ can be placed into one of $q-2m+2$ boxes, except the upper left and the lower right boxes.
Each of the numbers $i+1,\dots,j-1$ can be placed into one of $q-2m+1$ boxes, except the upper left and the lower right boxes and the box containing $(i,j)$.
Each of the numbers $j+1,\dots,k-1$ can be placed into one of $q-2m+2$ boxes, except the upper left and the lower right boxes.
Finally, each of the numbers $k+1,\dots,m$ can be placed into one of $q-2m+1$ boxes, except the upper left, the upper right, and the lower right boxes.
Therefore, in this case, the desired number is
\begin{align*}
(q-2m+1)^{m-i+j-k-1}(q-2m+2)^{i-j+k-1}.
\end{align*}

From the discussion above, by applying Lemma~\ref{hodai}, we have that
\begin{align*}
\left|M(\mathcal{A}^{(3)}_q)\right|
=&\,T^{j-i-1}(T+1)^{m+i-j}+\sum_{k=1}^{i-1}T^{j-i-1}(T+1)^{m+i-j-k}(T+2)^{k-1}\\
&+\sum_{k=i+1}^{j-1}T^{j-k-1}(T+1)^{m-i-j+k}(T+2)^{i-1}+\sum_{k=j+1}^{m}(T+1)^{m-i+j-k-1}(T+2)^{i-j+k-1}\\
=&\,T^{j-i-1}(T+1)^{m+i-j}+T^{j-i-1}(T+1)^{m-j+1}\left((T+2)^{i-1}-(T+1)^{i-1}\right)\\
&+(T+1)^{m-j+1}(T+2)^{i-1}\left((T+1)^{j-i-1}-T^{j-i-1}\right)\\
&+(T+1)^{j-i-1}(T+2)^{i}\left((T+2)^{m-j}-(T+1)^{m-j}\right)\\
=&\,(T+1)^{m-i}(T+2)^{i-1}+(T+1)^{j-i-1}(T+2)^{m+i-j}-(T+1)^{m-i-1}(T+2)^{i}\\
=&\,(T+1)^{j-i-1}(T+2)^{m+i-j}-(T+1)^{m-i-1}(T+2)^{i-1}\\
=&\,(T+1)^{j-i-1}(T+2)^{i-1}\left((T+2)^{m-j+1}-(T+1)^{m-j}\right),
\end{align*}
where $T:=q-2m$.

\section{Characteristic quasi-polynomial of restriction on $\{x_i-x_j=1\}$}\label{sec-quasi-poly-restxi=xj+1}

Let $\mathcal{A}^{(4)}=\mathcal{B}_m^{\{x_i-x_j=1\}}$ for $1\leq i<j \leq m$, and we prove the following theorem.
\begin{Them}\label{chara-quasi-del-by-xi=xj+1}
We have
\begin{align*}
\left|M(\mathcal{A}^{(4)}_q)\right|=(q-2m)^{m+i-j}(q-2m+1)^{j-i-1}
\end{align*}
for any $q\in\mathbb{Z}$ with $q\gg 0$.
\end{Them}
The complement $M(\mathcal{A}^{(4)}_q)$ is the set of $(x_1,\dots,x_m)\in\mathbb{Z}_q^m$ satisfying the following conditions:
\begin{align*}
&x_i=x_j+1,\\
&x_s\neq 0,\ x_s\neq 1\ (s\in [m]),\\
&x_s\neq x_t\ (s,t\in [m],\ s\neq t,\ \{s,t\}\neq \{i,j\}),\\
&x_s\neq x_t+1\ (s,t\in [m],\ s<t,\ (s,t)\neq (i,j)),\\
&x_s\neq -x_t,\ x_s\neq -x_t+1\ (s,t\in [m],\ s\neq t).
\end{align*}
We fix the condition $x_i=x_j+1$ and count the number of elements in $(x_1,\dots,x_m)\in M(\mathcal{A}^{(4)}_q)$ using a modified version of the counting method described in Section \ref{sec-count-method}.
The condition $x_i=x_j+1$ means that the circle labeled with $j$ is placed just in front of the circle labeled with $i$ in a clockwise direction.
Therefore, once the box containing the number $i$ is determined, there is no need to specify the box containing the number $j$.
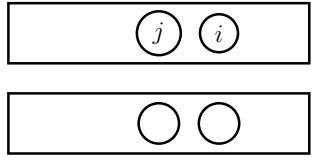
\begin{center}
\scalebox{0.8}{
\begin{tikzpicture}[main/.style = {draw, circle, very thick}] 
\draw [very thick] (0,0) rectangle (5,1); \draw [very thick] (0,1.5) rectangle (5,2.5);
\node[main](1) at (2.5,0.5) {{\color{white}\large $0$}}; \node[main](2) at (3.5,0.5) {{\color{white}\large $0$}}; \node[main](4) at (2.5,2) {{\large $j$}}; \node[main](5) at (3.5,2) {{\large $i$}};
\end{tikzpicture}
}
\end{center}

\subsection{The proof of Theorem \ref{chara-quasi-del-by-xi=xj+1} when $q$ is odd}\label{sec-quasi-poly-rest-xj=xi-1-odd}
We assume that $q$ is odd.
Making a modification to the procedure in Section \ref{sec-count-method-odd}, we create boxes and circles, and take the corresponding element $(x_1,\dots,x_m)\in M(\mathcal{A}^{(4)}_q)$.
\begin{itemize}
\item Prepare $q-2m+1$ boxes side by side, $\displaystyle \frac{q+1}{2}-m$ on the upper side and $\displaystyle \frac{q+1}{2}-m$ on the lower side.
Place each of the numbers $1,\dots,j-1,j+1,\dots,m$ into one of $q-2m$ boxes, except the upper left box.
The order to place the number is $i$ first, then the numbers from $1$ to $m$ that are neither $i$ nor $j$.
\item We create circles by the same procedure described in Section \ref{sec-count-method-odd}.
When placing each of the numbers $i+1,\dots,j-1$ into the same box that contains the number $i$, we have two choices either before $i$ or after $i$ in a clockwise direction.
If $s$ $(i+1\leq s\leq j-1)$ is before $i$, then the circle labeled with $s$ is placed before the circle labeled with $j$.
Since $x_j=x_s+1$ and $s<j$, no contradiction occurs.
If $s$ $(i+1\leq s\leq j-1)$ is after $i$, then the circle labeled with $s$ is placed after the circle labeled with $i$.
Since $x_s=x_i+1$ and $i<s$, no contradiction occurs.
We note that, in the box containing $i$, the circles labeled with $1,\dots,i-1$ are placed before the circle labeled with $j$ and the circles labeled with $j+1,\dots,m$ are placed after the circle labeled with $i$. 
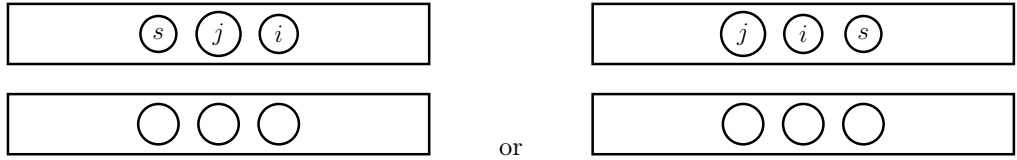
\begin{center}
\scalebox{0.8}{
\begin{tikzpicture}[main/.style = {draw, circle, very thick}] 
\draw [very thick] (0,0) rectangle (7,1); \draw [very thick] (0,1.5) rectangle (7,2.5);
\node[main](1) at (2.5,0.5) {{\color{white}\large $0$}}; \node[main](2) at (3.5,0.5) {{\color{white}\large $0$}}; \node[main](3) at (4.5,0.5) {{\color{white}\large $0$}};
\node[main](4) at (2.5,2) {{\large $s$}}; \node[main](5) at (3.5,2) {{\large $j$}}; \node[main](6) at (4.5,2) {{\large $i$}};
\end{tikzpicture}
}
\qquad or\ \qquad
\scalebox{0.8}{
\begin{tikzpicture}[main/.style = {draw, circle, very thick}] 
\draw [very thick] (0,0) rectangle (7,1); \draw [very thick] (0,1.5) rectangle (7,2.5);
\node[main](1) at (2.5,0.5) {{\color{white}\large $0$}}; \node[main](2) at (3.5,0.5) {{\color{white}\large $0$}}; \node[main](3) at (4.5,0.5) {{\color{white}\large $0$}};
\node[main](4) at (2.5,2) {{\large $j$}}; \node[main](5) at (3.5,2) {{\large $i$}}; \node[main](6) at (4.5,2) {{\large $s$}};
\end{tikzpicture}
}
\end{center}
\item When placing the circles, place the circle labeled with $j$ clockwise just in front of the circle labeled with $i$.
In addition, place an unlabeled circle on the opposite side of the circle labeled with $j$ since $x_j\neq -x_s$ for any $s\in[m]\setminus\{j\}$.
\end{itemize}
For example, if  $m=6$, $q=17$, $i=2$, and $j=5$, then the boxes and the numbers
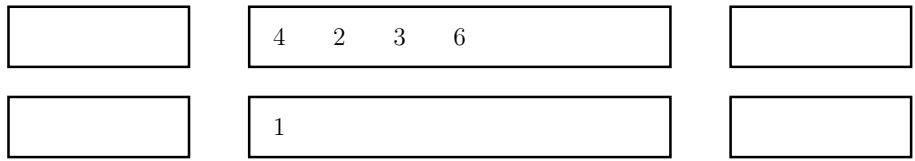
\begin{center}
\scalebox{0.8}{
\begin{tikzpicture}[main/.style = {draw, circle, very thick}] 
\draw [very thick] (0,0) rectangle (3,1); \draw [very thick] (4,0) rectangle (11,1); \draw [very thick] (12,0) rectangle (15,1);
\draw [very thick] (0,1.5) rectangle (3,2.5); \draw [very thick] (4,1.5) rectangle (11,2.5); \draw [very thick] (12,1.5) rectangle (15,2.5);
\node(1) at (4.5,2) {{\large $4$}}; \node(2) at (5.5,2) {{\large $2$}}; \node(3) at (6.5,2) {{\large $3$}}; \node(4) at (7.5,2) {{\large $6$}};
\node(1) at (4.5,0.5) {{\large $1$}};
\end{tikzpicture}
}
\end{center}
correspond to the boxes and circles
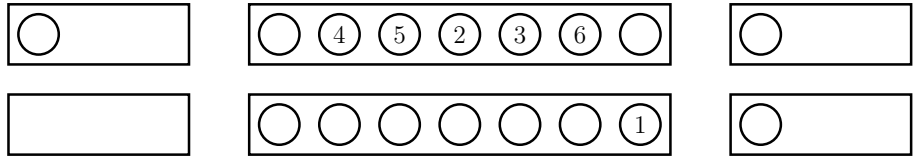
\begin{center}
\scalebox{0.8}{
\begin{tikzpicture}[main/.style = {draw, circle, very thick}] 
\draw [very thick] (0,0) rectangle (3,1); \draw [very thick] (4,0) rectangle (11,1); \draw [very thick] (12,0) rectangle (15,1);
\draw [very thick] (0,1.5) rectangle (3,2.5); \draw [very thick] (4,1.5) rectangle (11,2.5); \draw [very thick] (12,1.5) rectangle (15,2.5);
\node[main](0) at (0.5,2) {{\color{white}\large $0$}};
\node[main](1) at (4.5,2) {{\color{white}\large $0$}}; \node[main](2) at (5.5,2) {{\large $4$}};\node[main](3) at (6.5,2) {{\large $5$}}; \node[main](4) at (7.5,2) {{\large $2$}}; \node[main](5) at (8.5,2) {{\large $3$}}; \node[main](6) at (9.5,2) {{\large $6$}}; \node[main](7) at (10.5,2) {{\color{white}\large $0$}};
\node[main](8) at (12.5,2) {{\color{white}\large $0$}};
\node[main](16) at (4.5,0.5) {{\color{white}\large $0$}}; \node[main](15) at (5.5,0.5) {{\color{white}\large $0$}}; \node[main](14) at (6.5,0.5) {{\color{white}\large $0$}}; \node[main](13) at (7.5,0.5) {{\color{white}\large $0$}}; \node[main](12) at (8.5,0.5) {{\color{white}\large $0$}}; \node[main](11) at (9.5,0.5) {{\color{white}\large $0$}}; \node[main](10) at (10.5,0.5) {{\large $1$}};
\node[main](9) at (12.5,0.5) {{\color{white}\large $0$}};
\end{tikzpicture}
}
\end{center}
while this corresponds to the element $(x_1,x_2,x_3,x_4,x_5,x_6)=(10,4,5,2,3,6)$.

We count the number of elements $(x_1,\dots,x_m)\in M(\mathcal{A}^{(4)}_q)$.
Recall that we place the number $i$ first and then, place the numbers from $1$ to $m$ that are neither $i$ nor $j$.
There are $q-2m$ boxes that can contain the number $i$, except the upper left box.
Then there are $q-2m$ boxes that can contain the numbers $1,\dots,i-1$, except the upper left box.
For the numbers $i+1,\dots,j-1$, there are $q-2m+1$ choices, since we have $q-2m$ boxes except the upper left box and there are two choices for the box that contains the number $i$.
Each of the numbers $j+1,\dots,m$ can be placed anywhere except the upper left box, so $q-2m$ boxes can be chosen.
Therefore, we have that
\begin{align*}
\left|M(\mathcal{A}^{(4)}_q)\right|=(q-2m)^{m+i-j}(q-2m+1)^{j-i-1}.
\end{align*}

\subsection{The proof of Theorem \ref{chara-quasi-del-by-xi=xj+1} when $q$ is even}\label{sec-quasi-poly-rest-xj=xi-1-even}
In this subsection, we assume that $q$ is even.
We count the number of elements $(x_1,\dots,x_m)\in M(\mathcal{A}^{(4)}_q)$ by dividing the cases by whether there exists an index $k$ such that $\displaystyle x_k=\frac{q}{2}$ or not.

(i)\ Suppose that there is no index $k\in[m]$ such that $\displaystyle x_k=\frac{q}{2}$.
Prepare $q-2m$ boxes, $\displaystyle \frac{q}{2}-m$ on the upper side and $\displaystyle \frac{q}{2}-m$ on the lower side, and one unlabeled circle corresponding to the element $\displaystyle \frac{q}{2}\in\mathbb{Z}_q$ on the right side of the boxes.
Place each of the numbers $1,\dots,j-1,j+1,\dots,m$ into one of the $q-2m-1$ boxes, except the upper left box.
As in the case when $q$ is odd, when placing each of the numbers $i+1,\dots,j-1$ into the same box as $i$, we have two choices either before $i$ or after $i$ in a clockwise direction.
Then place the circle labeled with $j$ clockwise just in front of the circle labeled with $i$, while  an unlabeled circle on the opposite side of the circle labeled with $j$.

There are $q-2m-1$ boxes that can contain the number $i$ and the numbers $1,\dots,i-1,j+1,\dots,m$, except the upper left box.
For the numbers $i+1,\dots,j-1$, there are $q-2m$ choices, since we have $q-2m$ boxes except the upper left box and there are two choices for the box containing $i$.
Therefore, in this case, the desired number is
\begin{align*}
(q-2m-1)^{m+i-j}(q-2m)^{j-i-1}.
\end{align*}

(i\hspace{-0.5mm}i)\ Suppose that there exists an index $k\in[m]$ such that $\displaystyle x_k=\frac{q}{2}$.
Create boxes and circles according to the following rules, and take the corresponding element $(x_1,\dots,x_m)\in M(\mathcal{A}^{(4)}_q)$.
\begin{itemize}
\item Prepare $q-2m+2$ boxes, $\displaystyle \frac{q}{2}-(m-1)$ on the upper side and $\displaystyle \frac{q}{2}-(m-1)$ on the lower side, and one circle labeled with $k$ on the right side of the boxes.
Place each of the numbers that are neither $j$ nor $k$ into one of $q-2m+1$ boxes, except the upper left box.
The order to place the number is $i$ first, then the numbers from $1$ to $m$ that are not $i$, $j$, or $k$.
\item As in the case when $q$ is odd, when placing each of the numbers $i+1,\dots,j-1$ into the same box as $i$, we have two choices.
\end{itemize}
For example, if $m=5$, $q=16$, $i=2$, $j=4$, and $k=1$, then to the boxes and numbers
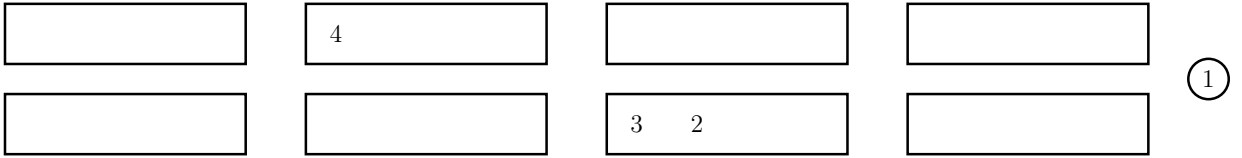
\begin{center}
\scalebox{0.8}{
\begin{tikzpicture}[main/.style = {draw, circle, very thick}] 
\draw [very thick] (0,0) rectangle (4,1); \draw [very thick] (5,0) rectangle (9,1); \draw [very thick] (10,0) rectangle (14,1); \draw [very thick] (15,0) rectangle (19,1);
\draw [very thick] (0,1.5) rectangle (4,2.5); \draw [very thick] (5,1.5) rectangle (9,2.5); \draw [very thick] (10,1.5) rectangle (14,2.5); \draw [very thick] (15,1.5) rectangle (19,2.5);
\node[main](0) at (20,1.25) {{\large $1$}};
\node(1) at (5.5,2) {{\large $4$}}; \node(2) at (10.5,0.5) {{\large $3$}};  \node(3) at (11.5,0.5) {{\large $2$}};
\end{tikzpicture}
}
\end{center}
correspond to the boxes and circles
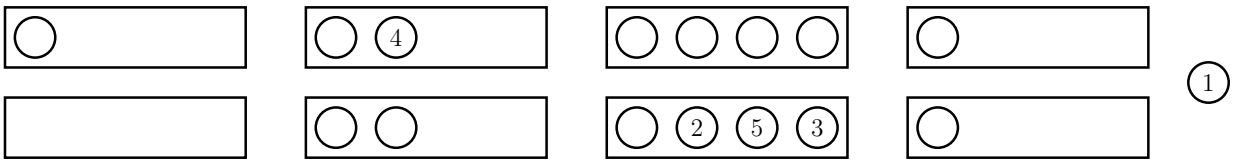
\begin{center}
\scalebox{0.8}{
\begin{tikzpicture}[main/.style = {draw, circle, very thick}] 
\draw [very thick] (0,0) rectangle (4,1); \draw [very thick] (5,0) rectangle (9,1); \draw [very thick] (10,0) rectangle (14,1); \draw [very thick] (15,0) rectangle (19,1);
\draw [very thick] (0,1.5) rectangle (4,2.5); \draw [very thick] (5,1.5) rectangle (9,2.5); \draw [very thick] (10,1.5) rectangle (14,2.5); \draw [very thick] (15,1.5) rectangle (19,2.5);
\node[main](8) at (20,1.25) {{\large $1$}};
\node[main](0) at (0.5,2) {{\color{white}\large $0$}};
\node[main](1) at (5.5,2) {{\color{white}\large $0$}}; \node[main](2) at (6.5,2) {{\large $4$}};
\node[main](3) at (10.5,2) {{\color{white}\large $0$}}; \node[main](4) at (11.5,2) {{\color{white}\large $0$}}; \node[main](5) at (12.5,2) {{\color{white}\large $0$}}; \node[main](6) at (13.5,2) {{\color{white}\large $0$}};
\node[main](7) at (15.5,2) {{\color{white}\large $0$}};
\node[main](15) at (5.5,0.5) {{\color{white}\large $0$}}; \node[main](14) at (6.5,0.5) {{\color{white}\large $0$}};
\node[main](13) at (10.5,0.5) {{\color{white}\large $0$}}; \node[main](12) at (11.5,0.5) {{\large $2$}}; \node[main](11) at (12.5,0.5) {{\large $5$}}; \node[main](10) at (13.5,0.5) {{\large $3$}};
\node[main](9) at (15.5,0.5) {{\color{white}\large $0$}};
\end{tikzpicture}
}
\end{center}
while this corresponds to the element $(x_1,x_2,x_3,x_4,x_5)=(8,12,10,2,11)$.
\begin{itemize}
\item For the same reasons discussed in Section \ref{sec-count-method-even}, the lower right box does not contain any numbers, since $x_k\neq -x_s+1$ for any $s\in [m]$ with $k\neq s$.
The upper right box does not contain any numbers larger than $k$, since $x_k\neq x_t+1$ for any $t\in [m]$ with $k<t$.
\end{itemize}

(i\hspace{-0.5mm}i-1)\ Suppose that $1\leq k\leq i-1$.
Since $i>k$, the number $i$ can be placed into one of $q-2m-1$ boxes, except the upper left, the upper right, and the lower right boxes.
There are $q-2m$ boxes that can contain the numbers $1,\dots,k-1$, except the upper left and the lower right boxes.
Then there are $q-2m-1$ boxes that can contain the numbers $k+1,\dots,i-1$, except the upper left, the upper right, and the lower right boxes.
Since each of $i+1,\dots,j-1$ cannot be placed in the upper left, the upper right, and the lower right boxes, and there are two choices when placing each of $i+1,\dots,j-1$ into the box containing $i$, there are $q-2m$ possible choices for each of the numbers $i+1,\dots,j-1$.
Finally, there are $q-2m-1$ boxes that can contain the numbers $j+1,\dots,m$, except the upper left, the upper right, and the lower right boxes.
Therefore, in this case, the desired number is
\begin{align*}
(q-2m-1)^{m+i-j-k}(q-2m)^{k-i+j-2}.
\end{align*}

(i\hspace{-0.5mm}i-2)\ Suppose that $k=i$.
In this case, since $x_i=x_j+1$, the circle labeled with $j$ is at the right edge of the upper right box.
Thus the numbers greater than $j$ cannot contain in the upper right box.
In addition, we note that there is only one choice when placing each of the numbers $i+1,\dots,j-1$ into the same box as $j$.
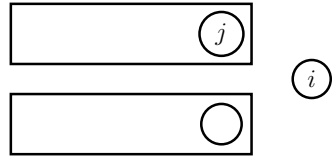
\begin{center}
\scalebox{0.8}{
\begin{tikzpicture}[main/.style = {draw, circle, very thick}] 
\draw [very thick] (0,0) rectangle (4,1); \draw [very thick] (0,1.5) rectangle (4,2.5);
\node[main](0) at (5,1.25) {{\large $i$}}; \node[main](1) at (3.5,2) {\large $j$}; \node[main](2) at (3.5,0.5) {{\color{white}\large $0$}};
\end{tikzpicture}
}
\end{center}
Therefore there are $q-2m$ boxes that can contain the numbers $1,\dots,i-1,i+1,\dots,j-1$, except the upper left and the lower right boxes.
Each of the numbers $j+1,\dots,m$ can be placed except the upper left, the upper right, and the lower right boxes, so $q-2m-1$ boxes can be chosen.
In this case, the desired number is
\begin{align*}
(q-2m-1)^{m-j}(q-2m)^{j-2}.
\end{align*}
This number is the same as the desired number in (i\hspace{-0.5mm}i-1), substituting $i$ for $k$.

(i\hspace{-0.5mm}i-3)\ Suppose that $i+1\leq k\leq j-1$.
Since $i<k$, the number $i$ can be placed into one of $q-2m$ boxes, except the upper left and the lower right boxes.
Also there are $q-2m$ boxes that can contain the numbers $1,\dots,i-1$, except the upper left and the lower right boxes.
When placing each of $i+1,\dots,j-1$ into the box containing $i$, we have two choices.
In addition, each of $i+1,\dots,k-1$ cannot be contained in the upper left and the lower right boxes, and each of $k+1,\dots,j-1$ cannot be contained in the upper left, the upper right, and the lower right boxes.
Thus there are $q-2m+1$ choices for each of the numbers $i+1,\dots,k-1$ and there are $q-2m$ choices for the numbers $k+1,\dots,j-1$.
Finally, there are $q-2m-1$ boxes that can contain the numbers $j+1,\dots,m$, except the upper left, the upper right, and the lower right boxes.
Therefore, in this case, the desired number is
\begin{align*}
(q-2m-1)^{m-j}(q-2m)^{i+j-k-1}(q-2m+1)^{k-i-1}.
\end{align*}

(i\hspace{-0.5mm}i-4)\ Suppose that $k=j$.
Since $\displaystyle x_i=x_j+1$ and $x_j=-x_j$, we have $x_i=-x_j+1$, which is a contradiction.

(i\hspace{-0.5mm}i-5)\ Suppose that $j+1\leq k\leq m$.
The number $i$ can be placed into one of $q-2m$ boxes, except the upper left and the lower right boxes.
There are $q-2m$ boxes that can contain the numbers $1,\dots,i-1$, except the upper left and the lower right boxes.
Each of $i+1,\dots,j-1$ cannot be placed into the upper left and the lower right boxes, while we have two choices when placing each of $i+1,\dots,j-1$ into the box containing $i$.
Thus there are $q-2m+1$ possible choices for each of the numbers $i+1,\dots,j-1$.
For the remaining numbers, there are $q-2m$ boxes that can contain the numbers $j+1,\dots,k-1$, except the upper left and the lower right boxes.
Finally, there are $q-2m-1$ boxes that can contain the numbers $k+1,\dots,m$, except the upper left, the upper right, and the lower right boxes.
Therefore, in this case, the desired number is
\begin{align*}
(q-2m-1)^{m-k}(q-2m)^{i-j+k-1}(q-2m+1)^{j-i-1}.
\end{align*}

From the discussion above, by applying Lemma~\ref{hodai}, we have that
\begin{align*}
\left|M(\mathcal{A}^{(4)}_q)\right|
=&\,(T-1)^{m+i-j}T^{j-i-1}+\sum_{k=1}^{i}(T-1)^{m+i-j-k}T^{k-i+j-2}\\
&+\sum_{k=i+1}^{j-1}(T-1)^{m-j}T^{i+j-k-1}(T+1)^{k-i-1}+\sum_{k=j+1}^{m}(T-1)^{m-k}T^{i-j+k-1}(T+1)^{j-i-1}\\
=&\,(T-1)^{m+i-j}T^{j-i-1}+(T-1)^{m-j}T^{j-i-1}\left(T^{i}-(T-1)^{i}\right)\\
&+(T-1)^{m-j}T^{i}\left((T+1)^{j-i-1}-T^{j-i-1}\right)+T^{i}(T+1)^{j-i-1}\left(T^{m-j}-(T-1)^{m-j}\right)\\
=&\,T^{m+i-j}(T+1)^{j-i-1},
\end{align*}
where $T:=q-2m$.

\section{Characteristic quasi-polynomial of restriction on $\{x_i+x_j=0\}$}\label{sec-quasi-poly-restxi=-xj}
Let $\mathcal{A}^{(5)}=\mathcal{B}_m^{\{x_i+x_j=0\}}$ for $1\leq i<j \leq m$, and we prove the following theorem.
\begin{Them}\label{chara-quasi-del-by-xi=-xj}
We have
\begin{align*}
\left|M(\mathcal{A}^{(5)}_q)\right|=
\begin{cases}
(q-2m)^{m-j}(q-2m+1)^{j-i}((q-2m+2)^{i-1}-(q-2m+1)^{i-2}) \;\;&(q\ \text{is odd}), \\
(q-2m)^{m-j+1}(q-2m+1)^{j-i-1}(q-2m+2)^{i-1} &(q\ \text{is even}).  
\end{cases} 
\end{align*}
\end{Them}
The complement $M(\mathcal{A}_q^{(5)})$ is the set of $(x_1,\dots,x_m)\in\mathbb{Z}_q^m$ satisfying the following conditions:
\begin{align*}
&x_j=-x_i,\\
&x_s\neq 0,\ x_s\neq 1\ (s\in [m]),\\
&x_s\neq x_t\ (s,t\in [m],\ s\neq t),\\
&x_s\neq x_t+1\ (s,t\in [m],\ s<t),\\
&x_s\neq -x_t,\ x_s\neq -x_t+1\ (s,t\in [m],\ s\neq t,\ \{s,t\}\neq \{i,j\}).
\end{align*}
We fix the condition $x_j=-x_i$ and count the number of elements in $(x_1,\dots,x_m)\in M(\mathcal{A}_q)$ using a modified version of the counting method described in Section \ref{sec-count-method}.
The condition $x_j=-x_i$ means that the circle labeled with $j$ is placed just on the opposite side of the circle labeled with $i$. 
Therefore, once the box containing the number $i$ is determined, there is no need to specify the box containing the number $j$.

\subsection{The proof of Theorem \ref{chara-quasi-del-by-xi=-xj} when $q$ is odd}
We assume that $q$ is odd. 
\begin{itemize}
\item Prepare $q-2m+3$ boxes side by side, $\displaystyle \frac{q+1}{2}-(m-1)$ on the upper side and $\displaystyle \frac{q+1}{2}-(m-1)$ on the lower side. 
Place each of the numbers $1,\dots,j-1,j+1,\ldots,m$ in one of $q-2m+2$ boxes, avoiding the upper left box. 
\item 
For the circle labeled with $i$, place the circle labeled with $j$ exactly on the opposite side of it. 
For the circles whose labels are not $i$, place the unlabeled circles on the opposite side of the labeled circles arranged. 
Then the number in each of the lower boxes is placed in descending order, starting from next of the circles placed in this way. 
Also rewrite each number as a circle labeled with the same number and place the unlabeled circles on the opposite side.
\end{itemize}
For example, if $m=5$, $q=13$, $i=2$ and $j=3$, then the boxes and the numbers 
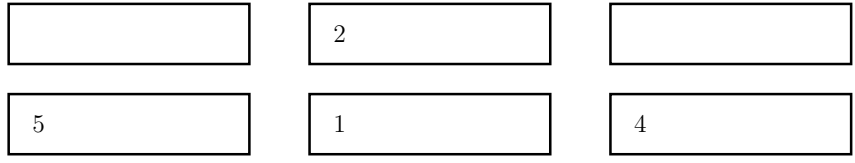
\begin{center}
\scalebox{0.8}{
\begin{tikzpicture}[main/.style = {draw, circle, very thick}] 
\draw [very thick] (0,0) rectangle (4,1); \draw [very thick] (5,0) rectangle (9,1); \draw [very thick] (10,0) rectangle (14,1);
\draw [very thick] (0,1.5) rectangle (4,2.5); \draw [very thick] (5,1.5) rectangle (9,2.5); \draw [very thick] (10,1.5) rectangle (14,2.5);
\node(1) at (5.5,2) {{\large $2$}}; \node(2) at (10.5,0.5) {{\large $4$}}; \node(3) at (0.5,0.5) {{\large $5$}}; \node(4) at (5.5,0.5) {{\large $1$}};
\end{tikzpicture}
}
\end{center}
correspond to the boxes and circles 
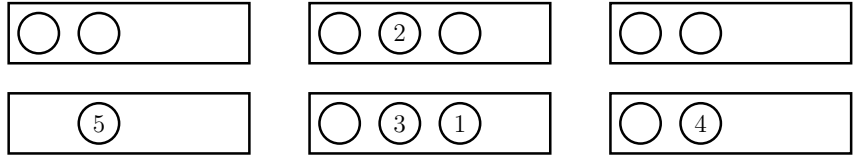
\begin{center}
\scalebox{0.8}{
\begin{tikzpicture}[main/.style = {draw, circle, very thick}] 
\draw [very thick] (0,0) rectangle (4,1); \draw [very thick] (5,0) rectangle (9,1); \draw [very thick] (10,0) rectangle (14,1);
\draw [very thick] (0,1.5) rectangle (4,2.5); \draw [very thick] (5,1.5) rectangle (9,2.5); \draw [very thick] (10,1.5) rectangle (14,2.5);
\node[main](1) at (0.5,2) {{\color{white}\large $0$}}; \node[main](2) at (1.5,2) {{\color{white}\large $0$}};
\node[main](3) at (5.5,2) {{\color{white}\large $0$}}; \node[main](4) at (6.5,2) {{\large $2$}};\node[main](5) at (7.5,2) {{\color{white}\large $0$}}; 
\node[main](7) at (10.5,2) {{\color{white}\large $0$}}; \node[main](8) at (11.5,2) {{\color{white}\large $0$}};
\node[main](9) at (1.5,0.5) {{\large $5$}};
\node[main](10) at (5.5,0.5) {{\color{white}\large $0$}}; \node[main](11) at (6.5,0.5) {{\large $3$}};\node[main](12) at (7.5,0.5) {{\large $1$}};
\node[main](14) at (10.5,0.5) {{\color{white}\large $0$}}; \node[main](15) at (11.5,0.5) {{\large $4$}};
\end{tikzpicture}
}
\end{center}
while this corresponds to the element $(x_1,x_2,x_3,x_4,x_5)=(9,3,10,7,12)$.

We count the number of elements $(x_1,\ldots,x_m) \in M(\mathcal{A}_q^{(5)})$. 
Then we should obey the following rules: 
\begin{itemize}
\item[(a)] The number $i$ cannot be placed in the lower left box. 
Indeed, if $i$ is placed there, then there is no number $k$ with $k>i$ next to $i$ placed in the same box; otherwise $x_j=-x_k+1$ holds, a contradiction. 
Even if there is no such number, i.e., $i$ is the leftmost number placed in the lower left box, then $x_j=1$, a contradiction. 
\item[(b)] The number $k$ with $k>i$ cannot be placed in the box that contains $i$. 
Indeed, if $k$ is placed there just next to $i$, then $x_j=-x_k+1$ holds, a contradiction. 
\item[(c)] The number $k$ with $k>j$ cannot be placed in the box that contains $j$, which is the opposite one of the box that contains $i$. 
Indeed, if $k$ is placed there just next to $j$, then $x_i=-x_k+1$ holds, a contradiction. 
\end{itemize}

\begin{figure}[h]
\begin{center}
\begin{minipage}{0.45\linewidth}
\scalebox{0.8}{
\begin{tikzpicture}[main/.style = {draw, circle, very thick}] 
\draw [very thick] (0,0) rectangle (4,1); \draw [very thick] (0,1.5) rectangle (4,2.5); 
\node[main] at (0.5,2) {{\color{white}\large $0$}}; \node[main] at (1.5,2) {{\color{white}\large $0$}}; \node[main] at (2.5,2) {{\large $j$}};
\node[main] at (1.5,0.5) {{\large $k$}};\node[main] at (2.5,0.5) {{\large $i$}};
\end{tikzpicture}}
or 
\scalebox{0.8}{
\begin{tikzpicture}[main/.style = {draw, circle, very thick}] 
\draw [very thick] (0,0) rectangle (4,1); \draw [very thick] (0,1.5) rectangle (4,2.5); 
\node[main] at (0.5,2) {{\color{white}\large $0$}}; \node[main] at (1.5,2) {{\large $j$}};
\node[main] at (1.5,0.5) {{\large $i$}};
\end{tikzpicture}}
\caption{(a)}
\end{minipage}
\begin{minipage}{0.25\linewidth}
\scalebox{0.8}{
\begin{tikzpicture}[main/.style = {draw, circle, very thick}] 
\draw [very thick] (0,0) rectangle (4,1); \draw [very thick] (0,1.5) rectangle (4,2.5); 
\node[main](2) at (2.5,2) {{\large $j$}};
\node[main](4) at (2.5,0.5) {{\large $i$}}; 
\node[main] at (1.5,0.5) {{\large $k$}};
\node[main] at (1.5,2) {{\color{white}\large $0$}};
\end{tikzpicture}}
\caption{(b)}
\end{minipage}
\begin{minipage}{0.25\linewidth}
\scalebox{0.8}{
\begin{tikzpicture}[main/.style = {draw, circle, very thick}] 
\draw [very thick] (0,0) rectangle (4,1); \draw [very thick] (0,1.5) rectangle (4,2.5); 
\node[main](2) at (2.5,2) {{\large $i$}};
\node[main](4) at (2.5,0.5) {{\large $j$}}; 
\node[main] at (1.5,0.5) {{\large $k$}};
\node[main] at (1.5,2) {{\color{white}\large $0$}};
\end{tikzpicture}}
\caption{(c)}
\end{minipage}
\end{center}
\end{figure}

(i)\ Suppose that there is no number $1,\ldots,i-1$ in the lower right box. 
There are $q-2m+1$ boxes that can contain the numbers $1,\ldots,i-1$, which are the boxes except the upper left and the lower right ones. 
In this case, the number $i$ cannot be placed in the lower right box. Indeed, if $i$ is contained in that box, then $i$ is the  rightmost number and $\displaystyle x_i=\frac{q+1}{2}$, so we have $\displaystyle x_j=-\frac{q+1}{2}=\frac{q-1}{2}$, a contradiction to $x_i \neq x_j+1$.
Hence, there are $q-2m$ boxes that can contain $i$, which are the boxes except the upper left, the lower right and the lower left (see (a)). 
There are $q-2m+1$ boxes that can contain the numbers $i+1,\ldots,j-1$, which are the boxes except the upper left and the one containing $i$ (see (b)). 
There are $q-2m$ boxes that can contain the numbers $j+1,\ldots,m$, which are the boxes except the upper left and the ones containins $i$ and $j$ (see (b) and (c)). 
Therefore the desired number is $$(q-2m+1)^{i-1} \cdot (q-2m) \cdot (q-2m+1)^{j-i-1} \cdot (q-2m)^{m-j}=(q-2m)^{m-j+1}(q-2m+1)^{j-2}.$$ 

(i\hspace{-0.5mm}i)\ Suppose one of the numbers $1,\ldots,i-1$ is in the lower right box. 
The total number to place $1,\ldots,i-1$ in this case is $(q-2m+2)^{i-1}-(q-2m+1)^{i-1}$. 
Unlike the above case, we can place $i$ in the lower right box since some number is placed at the right to $i$. 
Hence, there are $q-2m+1$ boxes that can contain $i$. 
Similarly to the above, there are $q-2m+1$ boxes that can contain the numbers $i+1,\ldots,j-1$, 
and there are $q-2m$ boxes that can contain the numbers $j+1,\ldots,m$. 
Therefore the desired number is $$(q-2m)^{m-j}(q-2m+1)^{j-i}((q-2m+2)^{i-1}-(q-2m+1)^{i-1}).$$

For the discussion above, we have that 
\begin{align*}
\left| M(\mathcal{A}_q^{(5)}) \right|&= T^{m-j+1}(T+1)^{j-2}+T^{m-j}(T+1)^{j-i}((T+2)^{i-1}-(T+1)^{i-1}) \\
&= T^{m-j}(T+1)^{j-i}((T+2)^{i-1}-(T+1)^{i-2}), 
\end{align*}
where $T:=q-2m$. 

\subsection{The proof of Theorem \ref{chara-quasi-del-by-xi=-xj} when $q$ is even}\label{sec-count-method-x_i+x_j=0-even}
In this subsection, we assume that $q$ is even. 
We count the number of elements $(x_1,\dots,x_m)\in M(\mathcal{A}_q^{(5)})$ by dividing the cases by whether there exists an index $k$ such that $\displaystyle x_k=\frac{q}{2}$ or not.

(i)\ Suppose that there is no index $k\in[m]$ such that $\displaystyle x_k=\frac{q}{2}$. 
Prepare $q-2m+2$ boxes, $\displaystyle \frac{q}{2}-(m-1)$ on the upper side and $\displaystyle \frac{q}{2}-(m-1)$ on the lower side, and one unlabeled circle corresponding to the element $\displaystyle \frac{q}{2} \in \mathbb{Z}_q$ on the right side of the boxes. 
Put each of the numbers $1,\ldots,j-1,j+1,\ldots,m$ into one of the $q-2m+1$ boxes, except the upper left box. 
As in the case when $q$ is odd, the number $i$ cannot be placed on the lower left box, and each of the numbers $i+1,\ldots,j-1$ cannot be placed in the box that contain $i$. Moreover, each of the numbers $j+1,\ldots,m$ can be placed in neither the box that contain $i$ nor the box that contain $j$. 
Therefore, in this case, the desired number is $$(q-2m-1)^{m-j}(q-2m)^{j-i}(q-2m+1)^{i-1}.$$ 

\bigskip

(i\hspace{-0.5mm}i)\ Suppose that there exists an index $k\in[m]$ such that $\displaystyle x_k=\frac{q}{2}$.
Create boxes and circles according to the following rules, and take the corresponding element $(x_1,\dots,x_m)\in M(\mathcal{A}_q^{(5)})$.
\begin{itemize}
\item Prepare $q-2m+4$ boxes, $\displaystyle \frac{q}{2}-(m-2)$ on the upper side and $\displaystyle \frac{q}{2}-(m-2)$ on the lower side, 
and one circle labeled with $k$ on the right side of the boxes. 
\item The lower right box does not contain any numbers. 
Indeed, if there is a labeled circle in the lower right box, then the clockwise next circle from the opposite circle of the rightmost circle, say labeled with $s$, is the circle labeled with $k$. 
In other words, $x_k=-x_s+1$ holds, a contradiction. 
\item The upper right box does not contain any numbers larger than $k$. 
Indeed, if the rightmost label on the upper right box is $s$ with $k<s$, then the circle clockwise preceding the circle labeled with $k$ has the label greater than $k$.
In other words, $x_k=x_s+1$ holds, a contradiction. 
\item The upper right box does not contain $i$. 
Indeed, if $i$ is placed there, since there is no number $s>i$ in the same box as the one containing $i$, $i$ is placed on the rightmost position of the upper right box.
Thus $j$ is placed just next to $k$.
This implies that $x_j=x_k+1=-x_k+1$, a contradiction. 
\end{itemize}

We divide the discussions into which range $k$ belongs. 
\begin{itemize}
\item[(i\hspace{-0.5mm}i-1)] Let $1 \leq k \leq i-1$. 
\begin{itemize}
\item There are $(q-2m+2)$ boxes that can contain the numbers $1,\ldots,k-1$, which are the boxes except the upper left and lower right ones. 
\item There are $(q-2m+1)$ boxes that can contain the numbers $k+1,\ldots,i-1$, which are the boxes except the upper left, the upper right and the lower right ones. 
\item There are $(q-2m)$ boxes that can contain $i$, which are the boxes except the upper left, the upper right, the lower right and the lower left ones (see (a)). 
\item There are $(q-2m)$ boxes that can contain the numbers $i+1,\ldots,j-1$, which are the boxes except the upper left, the upper right, the lower right and the one containing $i$ (see (b)). 
\item There are $(q-2m-1)$ boxes that can contain the numbers $j+1,\ldots,m$, which are the boxes except the upper left, the upper right, the lower right, and the ones containing $i$ and $j$ (see (b) and (c)). 
\end{itemize}
Hence, the desired number is $$(q-2m-1)^{m-j}(q-2m)^{j-i}(q-2m+1)^{i-k-1}(q-2m+2)^{k-1}$$ 
in the case of $1 \leq k \leq i-1$. 
\item[(i\hspace{-0.5mm}i-2)] Suppose $k=i$. Since $x_k=-x_k$ holds, we obtain $x_i=-x_i=x_j$, a contradiction. 
\item[(i\hspace{-0.5mm}i-3)] Let $i+1 \leq k \leq j-1$. 
\begin{itemize}
\item There are $(q-2m+2)$ boxes that can contain the numbers $1,\ldots,i-1$, which are the boxes except the upper left and lower right ones. 
\item There are $(q-2m)$ boxes that can contain $i$, which are the boxes except the upper left, the upper right, the lower right and the lower left ones (see (a)). 
\item There are $(q-2m+1)$ boxes that can contain the numbers $i+1,\ldots,k-1$, which are the boxes except the upper left, the lower right ones and the one containing $i$ (see (b)). 
\item There are $(q-2m)$ boxes that can contain the numbers $k+1,\ldots,j-1$, which are the boxes except the upper left, the upper right, the lower right ones and the one containing $i$ (see (b)). 
\item There are $(q-2m-1)$ boxes that can contain the numbers $j+1,\ldots,m$, which are the boxes except the upper left, the upper right and the lower right ones and the ones containing $i$ and $j$ (see (b) and (c)). 
\end{itemize}
Hence, the desired number is $$(q-2m-1)^{m-j}(q-2m)^{j-k}(q-2m+1)^{k-i-1}(q-2m+2)^{i-1}$$ 
in the case of $i+1 \leq k \leq j-1$. 
\item[(i\hspace{-0.5mm}i-4)] Let $j+1 \leq k \leq m$. 
\begin{itemize}
\item There are $(q-2m+2)$ boxes that can contain the numbers $1,\ldots,i-1$, which are the boxes except the upper left and lower right ones. 
\item There are $(q-2m)$ boxes that can contain $i$, which are the boxes except the upper left, the upper right, the lower right and the lower left ones (see (a)). 
\item There are $(q-2m+1)$ boxes that can contain the numbers $i+1,\ldots,j-1$, which are the boxes except the upper left, the lower right ones and the one containing $i$ (see (b)).
\item There are $(q-2m)$ boxes that can contain the numbers $j+1,\ldots,k-1$, which are the boxes except the upper left and the lower right ones and the ones containing $i$ and $j$ (see (b) and (c)).
\item There are $(q-2m-1)$ boxes that can contain the numbers $k+1,\ldots,m$, which are the boxes except the upper left, the upper right and the lower right ones and the ones containing $i$ and $j$ (see (b) and (c)).
\end{itemize}
Hence, the desired number is $$(q-2m-1)^{m-k}(q-2m)^{k-j}(q-2m+1)^{j-i-1}(q-2m+2)^{i-1}$$ 
in the case of $j+1 \leq k \leq m$. 
\end{itemize}

From the discussion above, by applying Lemma~\ref{hodai}, we have that
\begin{align*}
\left|M(\mathcal{A}_q^{(5)})\right|&=(T-1)^{m-j}T^{j-i}(T+1)^{i-1}+\sum_{k=1}^{i-1}(T-1)^{m-j}T^{j-i}(T+1)^{i-k-1}(T+2)^{k-1} \\
&+\sum_{k=i+1}^{j-1}(T-1)^{m-j}T^{j-k}(T+1)^{k-i-1}(T+2)^{i-1} + \sum_{k=j+1}^m(T-1)^{m-k}T^{k-j}(T+1)^{j-i-1}(T+2)^{i-1} \\
&=(T-1)^{m-j}T^{j-i}(T+1)^{i-1}+(T-1)^{m-j}T^{j-i}((T+2)^{i-1}-(T+1)^{i-1}) \\
&+(T-1)^{m-j}T(T+2)^{i-1}((T+1)^{j-i-1}-T^{j-i-1}) \\
&+T(T+1)^{j-i-1}(T+2)^{i-1}(T^{m-j}-(T-1)^{m-j}) \\
&=T^{m-j+1}(T+1)^{j-i-1}(T+2)^{i-1},  
\end{align*}
where $T:=q-2m$.

\section{Characteristic quasi-polynomial of restriction on $\{x_i+x_j=1\}$}\label{sec-quasi-poly-restxi=-xj+1}
Let $\mathcal{A}^{(6)}=\mathcal{B}_m^{\{x_i+x_j=1\}}$ for $1\leq i<j \leq m$, and we prove the following theorem.
\begin{Them}\label{chara-quasi-del-by-xi=-xj+1}
We have
\begin{align*}
\left|M(\mathcal{A}^{(6)}_q)\right|=
\begin{cases}
(q-2m)^{i-1}(q-2m+1)^{j-i}(q-2m+2)^{m-j} \;\;&(q\ \text{is odd}), \\
(q-2m)^{i-1}(q-2m+1)^{j-i-1}((q-2m+2)^{m-j+1}-(q-2m+1)^{m-j}) &(q\ \text{is even}). 
\end{cases}
\end{align*}
\end{Them}
The complement $M(\mathcal{A}_q^{(4)})$ is the set of $(x_1,\dots,x_m)\in\mathbb{Z}_q^m$ satisfying the following conditions:
\begin{align*}
&x_j=-x_i+1,\\
&x_s\neq 0,\ x_s\neq 1\ (s\in [m]),\\
&x_s\neq x_t\ (s,t\in [m],\ s\neq t),\\
&x_s\neq x_t+1\ (s,t\in [m],\ s<t),\\
&x_s\neq -x_t,\ x_s\neq -x_t+1\ (s,t\in [m],\ s\neq t,\ \{s,t\}\neq \{i,j\}).
\end{align*}
We fix the condition $x_j=-x_i+1$ and count the number of elements in $(x_1,\dots,x_m)\in M(\mathcal{A}^{(6)}_q)$ using a modified version of the counting method described in Section \ref{sec-count-method}.
The condition $x_i=-x_j+1$ means that the circle labeled with $j$ is clockwisely next circle from the opposite circle of the circle labeled with $i$.
Therefore, once the box containing the number $i$ is determined, there is no need to specify the box containing the number $j$.

\subsection{The proof of Theorem \ref{chara-quasi-del-by-xi=-xj+1} when $q$ is odd}
In this subsection, we assume that $q$ is odd. 
\begin{itemize}
\item Prepare $q-2m+1$ boxes side by side, $\displaystyle \frac{q+1}{2}-m$ on the upper side and $\displaystyle \frac{q+1}{2}-m$ on the lower side. Place each of the numbers $1,\dots,j-1,j+1,\ldots,m$ in one of $q-2m+1$ boxes. 
The numbers $1,\ldots,i-1$ should avoid the upper left box, while the other numbers do not need to avoid it. (The reason will be explained later.) 
\item 
For the circle labeled with $i$, place the circle labeled with $j$ on the next to the left (resp. right) of the opposite side of it if $i$ is placed on the upper (resp. lower) side. 
For the circles whose labels are not $i$, place the unlabeled circles on the opposite side of the labeled circles arranged. 
Then the number in each of the lower boxes is placed in descending order, starting from next of circles placed in this way. 
Also rewrite each number as a circle labeled with the same number and place the unlabeled circles on the opposite side.
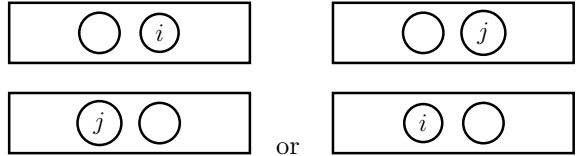
\begin{center}
\scalebox{0.8}{
\begin{tikzpicture}[main/.style = {draw, circle, very thick}] 
\draw [very thick] (0,0) rectangle (4,1); \draw [very thick] (0,1.5) rectangle (4,2.5); 
\node[main] at (1.5,2) {{\color{white}\large $0$}}; \node[main] at (2.5,2) {{\large $i$}}; 
\node[main] at (1.5,0.5) {{\large $j$}}; \node[main] at (2.5,0.5) {{\color{white}\large $0$}}; 
\end{tikzpicture}} \; or \; 
\scalebox{0.8}{
\begin{tikzpicture}[main/.style = {draw, circle, very thick}] 
\draw [very thick] (0,0) rectangle (4,1); \draw [very thick] (0,1.5) rectangle (4,2.5); 
\node[main] at (1.5,2) {{\color{white}\large $0$}}; \node[main] at (2.5,2) {{\large $j$}}; 
\node[main] at (1.5,0.5) {{\large $i$}}; \node[main] at (2.5,0.5) {{\color{white}\large $0$}}; 
\end{tikzpicture}}
\end{center}
\end{itemize}
For example, if $m=5$, $q=15$, $i=2$ and $j=3$, then the boxes and the number 
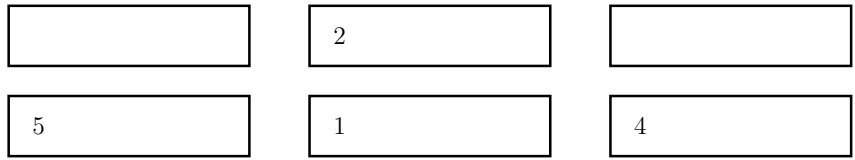
\begin{center}
\scalebox{0.8}{
\begin{tikzpicture}[main/.style = {draw, circle, very thick}] 
\draw [very thick] (0,0) rectangle (4,1); \draw [very thick] (5,0) rectangle (9,1); \draw [very thick] (10,0) rectangle (14,1);
\draw [very thick] (0,1.5) rectangle (4,2.5); \draw [very thick] (5,1.5) rectangle (9,2.5); \draw [very thick] (10,1.5) rectangle (14,2.5);
\node(1) at (5.5,2) {{\large $2$}}; \node(2) at (10.5,0.5) {{\large $4$}}; \node(3) at (0.5,0.5) {{\large $5$}}; \node(4) at (5.5,0.5) {{\large $1$}};
\end{tikzpicture}
}
\end{center}
correspond to the boxes and circles 
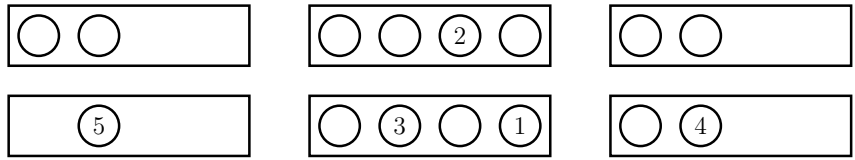
\begin{center}
\scalebox{0.8}{
\begin{tikzpicture}[main/.style = {draw, circle, very thick}] 
\draw [very thick] (0,0) rectangle (4,1); \draw [very thick] (5,0) rectangle (9,1); \draw [very thick] (10,0) rectangle (14,1);
\draw [very thick] (0,1.5) rectangle (4,2.5); \draw [very thick] (5,1.5) rectangle (9,2.5); \draw [very thick] (10,1.5) rectangle (14,2.5);
\node[main](1) at (0.5,2) {{\color{white}\large $0$}}; \node[main](2) at (1.5,2) {{\color{white}\large $0$}};
\node[main] at (5.5,2) {{\color{white}\large $0$}}; \node[main](3) at (6.5,2) {{\color{white}\large $0$}}; \node[main](4) at (7.5,2) {{\large $2$}};\node[main](5) at (8.5,2) {{\color{white}\large $0$}}; 
\node[main](7) at (10.5,2) {{\color{white}\large $0$}}; \node[main](8) at (11.5,2) {{\color{white}\large $0$}};
\node[main](9) at (1.5,0.5) {{\large $5$}};
\node[main](10) at (5.5,0.5) {{\color{white}\large $0$}}; \node[main](11) at (6.5,0.5) {{\large $3$}};\node[main](12) at (8.5,0.5) {{\large $1$}}; \node[main](5) at (7.5,0.5) {{\color{white}\large $0$}}; 
\node[main](14) at (10.5,0.5) {{\color{white}\large $0$}}; \node[main](15) at (11.5,0.5) {{\large $4$}};
\end{tikzpicture}
}
\end{center}
while this corresponds to the element $(x_1,x_2,x_3,x_4,x_5)=(10,4,12,8,14)$. 

We count the number of elements $(x_1,\ldots,x_m) \in M(\mathcal{A}_q^{(6)})$. 

\begin{itemize}
\item There are $q-2m$ boxes that can contain the numbers $1,\ldots,i-1$, which are the boxes except the upper left one. 
\item The number $i$ can be placed in the upper left box. 
Indeed, once $i$ is placed there, since $j$ is placed the lower left box, the unlabeled circle is placed in the upper left box at the second position from the left because no number less than $i$ is placed there. Hence, no contradiction occurs. 
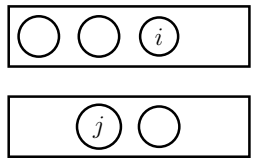
\begin{center}
\scalebox{0.8}{
\begin{tikzpicture}[main/.style = {draw, circle, very thick}] 
\draw [very thick] (0,0) rectangle (4,1); \draw [very thick] (0,1.5) rectangle (4,2.5); 
\node[main](2) at (0.5,2) {{\color{white}\large $0$}};
\node[main](2) at (1.5,2) {{\color{white}\large $0$}};
\node[main](2) at (2.5,2) {{\large $i$}};
\node[main](4) at (1.5,0.5) {{\large $j$}}; \node[main](2) at (2.5,0.5) {{\color{white}\large $0$}};
\end{tikzpicture}
}
\end{center}
Thus, there are $q-2m+1$ boxes that can contain $i$. 
\item 
If $i$ is placed in the upper left box, then the number $k$ with $i<k<j$ can be placed there.

If $i$ is not placed in the upper left box, then the number $k$ with $i<k<j$ cannot be placed in the upper left box, 
but instead, there are two ways to place $k$ in the box that contains $i$. 
Indeed, if $i$ is placed in some box in the upper (resp. lower) side, then we can place $k$ to the left (resp. right) of $i$ since there will be an unlabeled circle just to the next box to the left (resp. right) of $i$. 
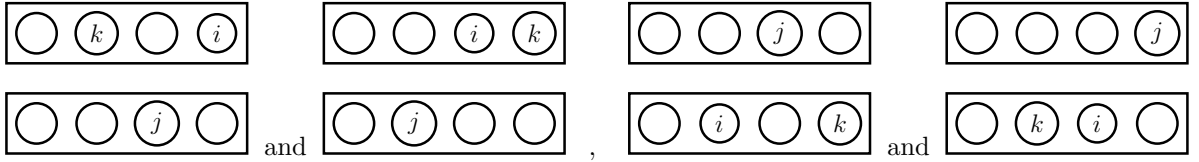
\begin{center}
\scalebox{0.8}{
\begin{tikzpicture}[main/.style = {draw, circle, very thick}] 
\draw [very thick] (0,0) rectangle (4,1); \draw [very thick] (0,1.5) rectangle (4,2.5); 
\node[main](2) at (0.5,2) {{\color{white}\large $0$}};
\node[main](2) at (1.5,2) {{\large $k$}};
\node[main](2) at (2.5,2) {{\color{white}\large $0$}};
\node[main](2) at (3.5,2) {{\large $i$}};
\node[main](2) at (0.5,0.5) {{\color{white}\large $0$}};
\node[main](2) at (1.5,0.5) {{\color{white}\large $0$}};
\node[main](2) at (2.5,0.5) {{\large $j$}};
\node[main](2) at (3.5,0.5) {{\color{white}\large $0$}};
\end{tikzpicture}
} and  
\scalebox{0.8}{
\begin{tikzpicture}[main/.style = {draw, circle, very thick}] 
\draw [very thick] (0,0) rectangle (4,1); \draw [very thick] (0,1.5) rectangle (4,2.5); 
\node[main](2) at (0.5,2) {{\color{white}\large $0$}};
\node[main](2) at (1.5,2) {{\color{white}\large $0$}};
\node[main](2) at (2.5,2) {{\large $i$}};
\node[main](2) at (3.5,2) {{\large $k$}};
\node[main](2) at (0.5,0.5) {{\color{white}\large $0$}};
\node[main](2) at (2.5,0.5) {{\color{white}\large $0$}};
\node[main](2) at (1.5,0.5) {{\large $j$}};
\node[main](2) at (3.5,0.5) {{\color{white}\large $0$}};
\end{tikzpicture}
}\; , \;
\scalebox{0.8}{
\begin{tikzpicture}[main/.style = {draw, circle, very thick}] 
\draw [very thick] (0,0) rectangle (4,1); \draw [very thick] (0,1.5) rectangle (4,2.5); 
\node[main](2) at (0.5,0.5) {{\color{white}\large $0$}};
\node[main](2) at (1.5,0.5) {{\large $i$}};
\node[main](2) at (2.5,0.5) {{\color{white}\large $0$}};
\node[main](2) at (3.5,0.5) {{\large $k$}};
\node[main](2) at (0.5,2) {{\color{white}\large $0$}};
\node[main](2) at (1.5,2) {{\color{white}\large $0$}};
\node[main](2) at (2.5,2) {{\large $j$}};
\node[main](2) at (3.5,2) {{\color{white}\large $0$}};
\end{tikzpicture}
} and  
\scalebox{0.8}{
\begin{tikzpicture}[main/.style = {draw, circle, very thick}] 
\draw [very thick] (0,0) rectangle (4,1); \draw [very thick] (0,1.5) rectangle (4,2.5); 
\node[main](2) at (0.5,2) {{\color{white}\large $0$}};
\node[main](2) at (1.5,2) {{\color{white}\large $0$}};
\node[main](2) at (2.5,2) {{\color{white}\large $0$}};
\node[main](2) at (3.5,2) {{\large $j$}};
\node[main](2) at (0.5,0.5) {{\color{white}\large $0$}};
\node[main](2) at (2.5,0.5) {{\large $i$}};
\node[main](2) at (1.5,0.5) {{\large $k$}};
\node[main](2) at (3.5,0.5) {{\color{white}\large $0$}};
\end{tikzpicture}
}
\end{center}

Hence, there are $q-2m+1$ boxes that can contain the numbers $i+1,\ldots,j-1$. 
\item Regarding the number $k$ with $k>j$, if $i$ is placed in the upper left (resp. lower left) box, then there are two ways to place $k$ in the lower left box, which contains $j$ (resp. $i$), by the similar reason to the above. 
In addition, we can place $k$ at the right side of the circle labeled with $i$ (resp. $j$) in the upper left box.
Note that the left side of the circle labeled with $i$ (resp. $j$) in the upper left box is forbidden for $k$ to be placed. 
Indeed, if $k$ is placed there, then we have $x_k=1$, a contradiction. 
If $i$ is not placed in the upper left box, then $k$ cannot be placed there, 
but instead, there are two ways to place the number $k$ with $k>j$ in the boxes that contains $i$ and $j$ by the similar reason to the above. 

Hence, there are $q-2m+2$ boxes that can contain the numbers $j+1,\ldots,m$. 
\end{itemize}

Therefore, we have that 
\begin{align*}
\left|M(\mathcal{A}_q^{(6)})\right|=T^{i-1}(T+1)^{j-i}(T+2)^{m-j}, 
\end{align*}
where $T:=q-2m$. 

\subsection{The proof of Theorem \ref{chara-quasi-del-by-xi=-xj+1} when $q$ is even}\label{sec-count-method-x_i+x_j=1-even}
In this subsection, we assume that $q$ is even.
We count the number of elements $(x_1,\dots,x_m)\in M(\mathcal{A}_q^{(6)})$ by dividing the cases by whether there exists an index $k$ such that $\displaystyle x_k=\frac{q}{2}$ or not.

(i)\ Suppose that there is no index $k\in[m]$ such that $\displaystyle x_k=\frac{q}{2}$.
Prepare $q-2m$ boxes, $\displaystyle \frac{q}{2}-m$ on the upper side and $\displaystyle \frac{q}{2}-m$ on the lower side, 
and one circle corresponding to the element $\displaystyle \frac{q}{2}\in\mathbb{Z}_q$ on the right side of the boxes. 
In this case, we see that the possible ways to place those numbers is almost the same as the case when $q$ is odd but we may just replace $q$ with $q-1$ because of the number of boxes. 
Therefore, in this case, the desired number is $$(q-2m-1)^{i-1}(q-2m)^{j-i}(q-2m+1)^{m-j}.$$ 

\bigskip

(i\hspace{-0.5mm}i)\ Suppose that there exists an index $k\in[m]$ such that $\displaystyle x_k=\frac{q}{2}$.
Create boxes and circles according to the following rules and take the corresponding element $(x_1,\dots,x_m)\in M(\mathcal{A}_q^{(6)})$.
\begin{itemize}
\item Prepare $q-2m+2$ boxes, $\displaystyle \frac{q}{2}-(m-1)$ on the upper side and $\displaystyle \frac{q}{2}-(m-1)$ on the lower side, 
and one circle labeled with $k$ on the right side of the boxes. 
\item The number $i$ can be placed in the lower right box if $j<k$.
Indeed, in this case, we have $x_k=x_j+1$ and $x_k\neq -x_i+1$, while no contradiction occurs.
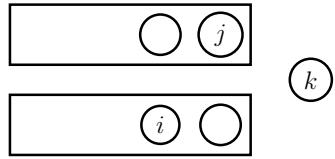
\begin{center}
\scalebox{0.8}{
\begin{tikzpicture}[main/.style = {draw, circle, very thick}] 
\draw [very thick] (0,0) rectangle (4,1); \draw [very thick] (0,1.5) rectangle (4,2.5);
\node[main](0) at (5,1.25) {{\large $k$}}; \node[main](1) at (3.5,2) {{\large $j$}}; \node[main](2) at (3.5,0.5) {{\color{white}\large $0$}}; \node[main](3) at (2.5,2) {{\color{white}\large $0$}}; \node[main](4) at (2.5,0.5) {{\large $i$}};
\end{tikzpicture}
}
\end{center}
Also, the numbers larger than $j$ can be placed in the lower right box if $k=i$.
(We will explain later in (i\hspace{-0.5mm}i-2).) 
The lower right box does not contain any numbers, except as previously noted. 
Indeed, if there is a labeled circle in the lower right box in such cases, then the clockwise next circle from the opposite circle of the rightmost circle, say labeled with $s$, is the circle labeled with $k$. 
In other words, $x_k=-x_s+1$ holds, a contradiction. 
\item The upper right box does not contain any numbers larger than $k$, except the case when $k=i$.
(We will explain later in (i\hspace{-0.5mm}i-2).) 
Indeed, if $k\neq i$ and the upper right box contains a number larger than $k$, then the rightmost circle in the upper right box is a labeled circle, say labeled with $s$ ($k<s$).
The circle clockwise preceding the circle labeled with $k$ has the label greater than $k$.
In other words, $x_k=x_s+1$ holds, a contradiction.
\end{itemize}
We divide the discussions into which range $k$ belongs. 
\begin{itemize}
\item[(i\hspace{-0.5mm}i-1)] Let $1 \leq k \leq i-1$. 
\begin{itemize}
\item There are $q-2m$ boxes that can contain the numbers $1,\ldots,k-1$, which are the boxes except the upper left and lower right ones. 
\item There are $q-2m-1$ boxes that can contain the numbers $k+1,\ldots,i-1$, which are the boxes except the upper left, the upper right and the lower right ones. 
\item Similarly to the case when $q$ is odd, 
there are $q-2m$ boxes that can contain $i, i+1,\ldots,j-1$. 
\item There are $q-2m+1$ boxes that can contain the numbers $j+1,\ldots,m$. 
\end{itemize}
Hence the desired number is $$(q-2m-1)^{i-k-1}(q-2m)^{j-i+k-1}(q-2m+1)^{m-j}$$ 
in the case of $1 \leq k \leq i-1$. 
\item[(i\hspace{-0.5mm}i-2)] Let $k=i$. Then $j$ is placed in the rightmost position of the lower right box, and there is an unlabeled circle in the rightmost position of the upper right box.
Thanks to this rightmost unlabeled circle, the numbers larger than $k=i$ can also be placed in the upper right box.
In addition, since the rightmost label on the lower right box is $j$, the numbers larger than $j$ can be placed in the lower right box.
Note that if $s$ ($s<j$) is contained in the lower right box and $s$ is the smallest number in it, then $x_s=x_j+1$ holds, a contradiction.
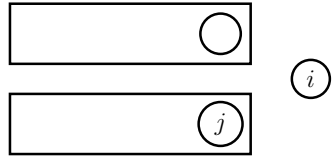
\begin{center}
\scalebox{0.8}{
\begin{tikzpicture}[main/.style = {draw, circle, very thick}] 
\draw [very thick] (0,0) rectangle (4,1); \draw [very thick] (0,1.5) rectangle (4,2.5);
\node[main](0) at (5,1.25) {{\large $i$}}; \node[main](1) at (3.5,2) {{\color{white}\large $0$}}; \node[main](2) at (3.5,0.5) {{\large $j$}};
\end{tikzpicture}
}
\end{center}
\begin{itemize}
\item There are $q-2m$ boxes that can contain the numbers $1,\ldots,i-1$, which are the boxes except the upper left and the lower right ones. 
\item Similarly, there are $q-2m$ boxes that can contain the numbers $i+1,\ldots,j-1$, which are the boxes except the upper left and the lower right ones.
\item There are $q-2m+1$ boxes that can contain the numbers $j+1,\ldots,m$, which are the boxes except the upper left one.
\end{itemize}
Hence the desired number is $$(q-2m)^{j-2}(q-2m+1)^{m-j}$$ 
in the case of $k=i$. 
\item[(i\hspace{-0.5mm}i-3)] Let $i+1 \leq k \leq j-1$. 
\begin{itemize}
\item There are $q-2m$ boxes that can contain the numbers $1,\ldots,i-1$, which are the boxes except the upper left and lower right ones. 
\item There are $q-2m+1$ boxes that can contain $i$. 
Similarly, there are $q-2m+1$ boxes that can contain the numbers $i+1,\ldots,k-1$. 
\item There are $q-2m$ boxes that can contain the numbers $k+1,\ldots,j-1$. 
\item There are $q-2m+1$ boxes that can contain the numbers $j+1,\ldots,m$. 
\end{itemize}
Hence the desired number is $$(q-2m)^{i+j-k-2}(q-2m+1)^{m-i-j+k}$$ 
in the case of $i+1 \leq k \leq j-1$. 
\item[(i\hspace{-0.5mm}i-4)] Suppose $k=j$. Since $x_k=-x_k$ holds, we obtain $x_i=-x_j+1=x_j+1$, a contradiction to $x_i-x_j \neq 1$. 
\item[(i\hspace{-0.5mm}i-5)] Let $j+1 \leq k \leq m$. 
\begin{itemize}
\item There are $q-2m$ boxes that can contain the numbers $1,\ldots,i-1$. 
\item There are $q-2m+2$ boxes that can contain $i$, since $i$ can be placed in the lower right box. 
\item There are $q-2m+1$ boxes that can contain the numbers $i+1,\ldots,j-1$. 
\item There are $q-2m+2$ boxes that can contain the numbers $j+1,\ldots,k-1$. 
\item There are $q-2m+1$ boxes that can contain the numbers $k+1,\ldots,m$. 
\end{itemize}
Hence the desired number is $$(q-2m)^{i-1}(q-2m+1)^{m-i+j-k-1}(q-2m+2)^{k-j}$$ 
in the case of $j+1 \leq k \leq m$. 
\end{itemize}

From the discussion above, by applying Lemma~\ref{hodai}, we have that
\begin{align*}
\left|M(\mathcal{A}_q^{(6)})\right|&=(T-1)^{i-1}T^{j-i}(T+1)^{m-j}+\sum_{k=1}^{i-1}(T-1)^{i-k-1}T^{j-i+k-1}(T+1)^{m-j}\\
&+\sum_{k=i}^{j-1}T^{i+j-k-2}(T+1)^{m-i-j+k} + \sum_{k=j+1}^m T^{i-1}(T+1)^{m-i+j-k-1}(T+2)^{k-j} \\
&=(T-1)^{i-1}T^{j-i}(T+1)^{m-j}+T^{j-i}(T+1)^{m-j}(T^{i-1}-(T-1)^{i-1}) \\
&+T^{i-1}(T+1)^{m-j}((T+1)^{j-i}-T^{j-i}) \\
&+T^{i-1}(T+1)^{j-i-1}(T+2)((T+2)^{m-j}-(T+1)^{m-j}) \\
&=T^{i-1}(T+1)^{j-i-1}((T+2)^{m-j+1}-(T+1)^{m-j}), 
\end{align*}
where $T:=q-2m$. 

\section{Period collapse in the characteristic quasi-polynomial of deletion of $\mathcal{B}_m$}

In this section, we prove Corollaries~\ref{period-collapse-1} and \ref{period-collapse-2} by using Theorem~\ref{thm:main}. 
As mentioned in Section~\ref{subsec:equiv}, we may check the period collapse of the restriction instead of the deletion. 

\begin{proof}[Proof of Corollary~\ref{period-collapse-1}]
If $H=\{x_i=0\}$ ($1 \leq i \leq m$) or $H=\{x_i=1\}$ ($1 \leq i \leq m$) or $H=\{x_i-x_j=1\}$ ($1 \leq i < j \leq m)$, then $|M(\mathcal{B}_m^H)_q|$ becomes a polynomial as Theorem~\ref{thm:main} shows. 

Let $H=\{x_i-x_j=0\}$ for $1 \leq i<j \leq m$. 
Then $|M(\mathcal{B}_m^H)_q|$ becomes a polynomial if and only if \begin{align*}
    U^{j-i-1}(U+1)^{m-j}((U+1)^i-U^{i-1}) = U^{j-i-1}(U+1)^{i-1}((U+1)^{m-j+1}-U^{m-j}),  
\end{align*}
where $U:=q-2m+1$. Thus, we conclude that 
\begin{align*}
&U^{j-i-1}(U+1)^{m-j}((U+1)^i-U^{i-1}) = U^{j-i-1}(U+1)^{i-1}((U+1)^{m-j+1}-U^{m-j}) \\
\Longleftrightarrow \;\; &(U+1)^{m+i-j}-U^{i-1}(U+1)^{m-j} = (U+1)^{m+i-j}-U^{m-j}(U+1)^{i-1} \\
\Longleftrightarrow \;\; &U^{i-1}(U+1)^{m-j} = U^{m-j}(U+1)^{i-1} 
\;\; \Longleftrightarrow \;\; i-1 = m-j
\;\; \Longleftrightarrow \;\; i+j = m+1.
\end{align*}

Let $H=\{x_i+x_j=0\}$ for $1 \leq i<j \leq m$. Then $|M(\mathcal{B}_m^H)_q|$ becomes a polynomial if and only if \begin{align*}
&T^{m-j}(T+1)^{j-i}((T+2)^{i-1}-(T+1)^{i-2}) = T^{m-j+1}(T+1)^{j-i-1}(T+2)^{i-1} \\
\Longleftrightarrow \;\; &(T+1)(T+2)^{i-1}-(T+1)^{i-1} = T(T+2)^{i-1} \;\;
\Longleftrightarrow \;\; (T+2)^{i-1} = (T+1)^{i-1} \;\;
\Longleftrightarrow \;\; i=1, 
\end{align*}
where $T:=q-2m$. 

Let $H=\{x_i+x_j=1\}$ for $1 \leq i<j \leq m$. Then $|M(\mathcal{B}_m^H)_q|$ becomes a polynomial if and only if \begin{align*}
&T^{i-1}(T+1)^{j-i}(T+2)^{m-j}=T^{i-1}(T+1)^{j-i-1}((T+2)^{m-j+1}-(T+1)^{m-j}) \\
\Longleftrightarrow \;\; &(T+1)(T+2)^{m-j} = (T+2)^{m-j+1}-(T+1)^{m-j} \;\;
\Longleftrightarrow \;\; (T+1)^{m-j} = (T+2)^{m-j} \;\;\Longleftrightarrow \;\; j=m. 
\end{align*}
\end{proof}
\begin{proof}[Proof of Corollary~\ref{period-collapse-2}]
Let $\mathcal{A}=\mathcal{B}_m$.
If $H$ and $H^{\prime}$ in $\mathcal{A}$ are parallel each other, 
since $(\mathcal{A}\setminus \{H\})^{H^{\prime}}=\mathcal{A}^{H^{\prime}}$, we see that \[
|M(\mathcal{A}_q)|=|M(\mathcal{A}_q\setminus \{H \})| - |M(\mathcal{A}_q^H )|=|M(\mathcal{A}_q\setminus \{H,H^{\prime}\})| - |M(\mathcal{A}_q^H)| - |M(\mathcal{A}_q^{{H^{\prime}}})|. 
\]
Hence, period collapse in $|M(\mathcal{A}_q\setminus \{H,H^{\prime}\})|$ is equivalent to period collapse in $|M(\mathcal{A}_q^{H})| + |M(\mathcal{A}_q^{H^{\prime}})|$ when period collapse occurs in $|M(\mathcal{A}_q)|$. 

In the case of the pair $H=\{x_i=0\}$ and $H^{\prime}=\{x_i=1\}$ for $1 \leq i \leq m$, it is clear that $|M(\mathcal{A}_q^{H})| + |M(\mathcal{A}_q^{H^{\prime}})|$ becomes a polynomial.  

In the case of the pair $H=\{x_i-x_j=0\}$ and $H^{\prime}=\{x_i-x_j=1\}$ for $1 \leq i<j \leq m$, 
since $|M(\mathcal{A}_q^{H^{\prime}})|$ is a polynomial, the desired condition is $i+j=m+1$ as Corollary~\ref{period-collapse-1} shows. 

In the case of the pair $H=\{x_i+x_j=0\}$ and $H^{\prime}=\{x_i+x_j=1\}$ for $1 \leq i<j \leq m$, we see the following: 
\begin{align*}
T^{m-j}(&T+1)^{j-i}((T+2)^{i-1}-(T+1)^{i-2})+T^{i-1}(T+1)^{j-i}(T+2)^{m-j} \\
&= T^{m-j+1}(T+1)^{j-i-1}(T+2)^{i-1}+T^{i-1}(T+1)^{j-i-1}((T+2)^{m-j+1}-(T+1)^{m-j}) \\
\Longleftrightarrow \;\; T^{m-j}(&T+1)((T+2)^{i-1}-(T+1)^{i-2})+T^{i-1}(T+1)(T+2)^{m-j} \\
&= T^{m-j+1}(T+2)^{i-1}+T^{i-1}((T+2)^{m-j+1}-(T+1)^{m-j}) \\
\Longleftrightarrow \;\; T^{m-j}(&(T+2)^{i-1}-(T+1)^{i-1})=T^{i-1}((T+2)^{m-j}-(T+1)^{m-j}) \\
\Longleftrightarrow \;\;\; m-j&=i-1
\;\;\Longleftrightarrow \;\; i+j=m+1, 
\end{align*}
as desired. 
\end{proof}


\begin{thebibliography}{10}
\bibitem{Athanasiadis1996} C. A. Athanasiadis. Characteristic polynomials of subspace arrangements and finite fields. \textit{Adv. Math.}, {\bf 122}, 193--233, (1996).
\bibitem{HTY2023} A. Higashitani, T. N. Tran and M. Yoshinaga. Period Collapse in Characteristic Quasi-Polynomials of Hyperplane Arrangements. \textit{Int. Math. Res. Not. IMRN}, {\bf 10}, 8934--8963, (2023).
\bibitem{KTT08} H. Kamiya, A. Takemura, and H. Terao. Periodicity of hyperplane arrangements with integral coefficients modulo positive integers. \textit{J. Algebr. Comb.}, {\bf 27} (3), 317--330, (2008).
\bibitem{KTT11} H. Kamiya, A. Takemura, and H. Terao. Periodicity of non-Central integral arrangements modulo positive integers. \textit{Ann. Comb.}, {\bf 15} (3), 449--464, (2011).
\bibitem{Liu-Tran-Yoshinaga21} Y. Liu, T. N. Tran, and M. Yoshinaga. $G$-Tutte polynomials and abelian Lie group arrangements. \textit{Int. Math. Res. Not. IMRN}, {\bf 1}, 150--188 (2021).
\bibitem{MN2024} Y. Mori and N. Nakashima. Characteristic quasi-polynomials for deformations of Coxeter arrangements of types A, B, C, and D. arXiv:2208.00735. 
\bibitem{PostnikovStanley} A. Postnikov and R. Stanley, Deformations of Coxeter hyperplane arrangements. \textit{J. Combin. Theory Ser. A}, {\bf 91}, 544--597, (2000).
\bibitem{Tamura23} S. Tamura, Postnikov–Stanley Linial arrangement conjecture. \textit{J. Algebr. Comb.} {\bf 58}, 651--679, (2023). 
\bibitem{Yoshinaga2018a} M. Yoshinaga. Characteristic polynomials of Linial arrangements for exceptional root systems. \textit{J. Combin.
Theory Ser. A}, {\bf 157}, 267--286, (2018).
\bibitem{Yoshinaga2018} M. Yoshinaga. Worpitzky partitions for root systems and characteristic quasi-polynomials. \textit{Tohoku Math. J.}, {\bf 70}, 39--63, (2018).
\end{thebibliography}
\end{document}